\documentclass{article}

\usepackage{color}

\usepackage{latexsym}
\usepackage{amssymb}
\usepackage{amsthm}
\usepackage{amssymb,amsmath}
\usepackage{MnSymbol}
\usepackage{mathrsfs}
\usepackage{dsfont}

\newtheorem{thm}{Theorem}
\newtheorem{lem}[thm]{Lemma}
\newtheorem{cor}[thm]{Corollary}
\newtheorem{prop}[thm]{Proposition}

\theoremstyle{definition}
\newtheorem{rem}[thm]{Remark}
\newtheorem{nota}[thm]{Notation}
\newtheorem{defi}[thm]{Definition}
\newtheorem{defi+nota}[thm]{Definition and Notation}

\title{3-uniform hypergraphs: modular decomposition and realization by tournaments}

\author{Abderrahim Boussa\"{\i}ri\footnotemark[1] \footnotemark[4]
\and 
Brahim Chergui \footnotemark[1] \footnotemark[3]
\and 
Pierre Ille\footnotemark[2] \footnotemark[5]
\and Mohamed Zaidi\footnotemark[1] \footnotemark[6]
}

\begin{document}

\maketitle

\footnotetext[1]{Facult\'e des Sciences A\"{\i}n Chock, 
D\'epartement de Math\'ematiques et Informatique, Km 8 route d'El Jadida, 
BP 5366 Maarif, Casablanca, Maroc}
\footnotetext[2]{Aix Marseille Univ, CNRS, Centrale Marseille, I2M, Marseille, France}
\footnotetext[4]{{\tt aboussairi@hotmail.com}}
\footnotetext[3]{{\tt cherguibrahim@yahoo.fr}}
\footnotetext[5]{{\tt pierre.ille@univ-amu.fr}}
\footnotetext[6]{{\tt zaidi.fsac@gmail.com}}

\begin{abstract}
Let $H$ be a 3-uniform hypergraph. A tournament $T$ defined on $V(T)=V(H)$ is a realization of $H$ if the edges of $H$ are exactly the 3-element subsets of $V(T)$ that induce 3-cycles. 
We characterize the 3-uniform hypergraphs that admit realizations by using a suitable modular decomposition. 
\end{abstract}

\medskip


\medskip


\section{Introduction}

Let $H$ be a 3-uniform hypergraph. 
A tournament $T$, with the same vertex set as $H$, is a realization of $H$ if the edges of $H$ are exactly the 3-element subsets of the vertex set of $T$ that induce 3-cycles. 
The aim of the paper is to characterize the 3-uniform hypergraphs that admit realizations (see \cite[Problem~1]{BILT04}). 
This characterization is comparable to that of the comparability graphs, that is, the graphs admitting a transitive orientation (see~\cite{GH}). 

In Section~\ref{section_tournament}, we recall some of the classic results on modular decomposition of tournaments. 

In the section below, we introduce a new notion of module for hypergraphs. 
We introduce also the notion of a modular covering, which generalizes the notion of a partitive family. 
In Subsection~\ref{subs Modular covering}, we show that the set of the modules of a hypergraph induces a modular covering. 
In Subsection~\ref{subs Gallai's decomposition}, we consider the notion of a strong module, which is the usual strengthening of the notion of a module (for instance, see Subsection~\ref{subs Modular decomposition of tournaments} for tournaments). 
We establish the analogue of Gallai's modular decomposition theorem for hypergraphs.

Let $H$ be a realizable 3-uniform hypergraph. 
Clearly, the modules of the realizations of $H$ are modules of $H$ as well, but the converse is false. 
Consider a realization $T$ of $H$. 
In Section~\ref{real_dec}, we characterize the modules of $H$ that are not modules of $T$. 
We deduce that a realizable 3-uniform hypergraph and its realizations share the same strong modules. 
Using Gallai's modular decomposition theorem, we prove that a realizable 3-uniform hypergraph is prime 
(i.e. all its modules are trivial) if and only if each of its realizations is prime too. 
We have similar results when we consider a comparability graph and its transitive orientations (for instance, see 
\cite[Theorem~3]{IR06} and \cite[Corollary~1]{IR06}).

In Section~\ref{real}, by using the modular decomposition tree, we demonstrate that a 3-uniform hypergraph is realizable if and only if all its prime, 3-uniform and induced subhypergraphs are realizable. 
We pursue by characterizing the prime and 3-uniform hypergraphs that are realizable. 
Hence \cite[Problem~1]{BILT04} is solved. 
We conclude by counting the realizations of a realizable 3-uniform hypergraph by using the modular decomposition tree. 
We have an analogous counting when we determine the number of transitive orientations of a comparability graph by using the modular decomposition tree of the comparability graph. 
The number of transitive orientations of a comparability graph was determined by Filippov and Shevrin~\cite{FS}. 
They used the notion of a saturated module, which is close to that of a strong module. 

At present, we formalize our presentation. 
We consider only finite structures. 
A hypergraph $H$ is defined by a vertex set $V(H)$ and an edge set $E(H)$, where 
$E(H)\subseteq 2^{V(H)}\setminus\{\emptyset\}$. 
In the sequel, we consider only hypergraphs $H$ such that 
\begin{equation*}
E(H)\subseteq 2^{V(H)}\setminus(\{\emptyset\}\cup\{\{v\}:v\in V(H)\}). 
\end{equation*}
Given $k\geq 2$, a hypergraph $H$ is {\em $k$-uniform} if 
\begin{equation*}
E(H)\subseteq\binom{V(H)}{k}. 
\end{equation*}
A hypergraph $H$ is {\em empty} if $E(H)=\emptyset$. 
Let $H$ be a hypergraph. 
With each $W\subseteq V(H)$, we associate the {\em subhypergraph} $H[W]$ of $H$ induced by $W$, which is defined by $V(H[W])=W$ and $E(H[W])=\{e\in E(H):e\subseteq W\}$. 

\begin{defi}\label{defi_module_hyper}
Let $H$ be a hypergraph. 
A subset $M$ of $V(H)$ is a {\em module} of $H$ if for each $e\in E(H)$ such that $e\cap M\neq\emptyset$ and 
$e\setminus M\neq\emptyset$, there exists $m\in M$ such that $e\cap M=\{m\}$ and for every $n\in M$, we have 
$$(e\setminus\{m\})\cup\{n\}\in E(H).$$ 
\end{defi}

\begin{nota}\label{module}
Given a hypergraph $H$, the set of the modules of $H$ is denoted by $\mathscr{M}(H)$. 
For instance, if $H$ is an empty hypergraph, then $\mathscr{M}(H)=2^{V(H)}$. 
\end{nota}

We study the set of the modules of a hypergraph. 
Let $S$ be a set. 
A family $\mathscr{F}$ of subsets of $S$ is a {\em partitive family} \cite[Definition~6]{CHM81} on $S$ if it satisfies the following assertions.  
\begin{itemize}
\item $\emptyset\in\mathscr{F}$, $S\in\mathscr{F}$, and for every $x\in S$, 
$\{x\}\in\mathscr{F}$.
\item For any $M,N\in\mathscr{F}$, $M\cap N\in\mathscr{F}$.
\item For any $M,N\in\mathscr{F}$, if $M\cap N\neq\emptyset$, 
$M\setminus N\neq\emptyset$ and $N\setminus M\neq\emptyset$, then 
$M\cup N\in\mathscr{F}$ and $(M\setminus N)\cup(N\setminus M)\in\mathscr{F}$.
\end{itemize}

We generalize the notion of a partitive family as follows. 

\begin{defi}\label{covering}
Let $S$ be a set. 
A {\em modular covering} of $S$ is a function $\mathfrak{M}$ which associates with each 
$W\subseteq S$ a set $\mathfrak{M}(W)$ of subsets of $W$, and which satisfies the following assertions. 
\begin{enumerate}
\item[(A1)] For each $W\subseteq S$, $\mathfrak{M}(W)$ is a partitive family on $W$. 
\item[(A2)] For any $W,W'\subseteq S$, if $W\subseteq W'$, then 
$$\{M'\cap W:M'\in\mathfrak{M}(W')\}\subseteq\mathfrak{M}(W).$$
\item[(A3)] For any $W,W'\subseteq S$, if $W\subseteq W'$ and $W\in\mathfrak{M}(W')$, then 
$$\{M'\in\mathfrak{M}(W'):M'\subseteq W\}=\mathfrak{M}(W).$$
\item[(A4)] Let $W,W'\subseteq S$ such that $W\subseteq W'$. 
For any $M\in\mathfrak{M}(W)$ and $M'\in\mathfrak{M}(W')$, if $M\cap M'=\emptyset$ and 
$M'\cap W\neq\emptyset$, then $M\in\mathfrak{M}(W\cup M')$. 
\item[(A5)] Let $W,W'\subseteq S$ such that $W\subseteq W'$. 
For any $M\in\mathfrak{M}(W)$ and $M'\in\mathfrak{M}(W')$, if $M\cap M'\neq\emptyset$, then $M\cup M'\in\mathfrak{M}(W\cup M')$. 
\end{enumerate}
\end{defi}

We obtain the following result.

\begin{prop}\label{properties_modules}
Given a hypergraph $H$, the function defined on $2^{V(H)}$, which maps each $W\subseteq V(H)$ to 
$\mathscr{M}(H[W])$, is a modular covering of $V(H)$. 
\end{prop}

Let $H$ be a hypergraph. 
By Proposition~\ref{properties_modules}, $\emptyset$, $V(H)$ and $\{v\}$, where $v\in V(H)$, are modules of $H$, called {\em trivial}. 
A hypergraph $H$ is {\em indecomposable} if all its modules are trivial, otherwise it is {\em decomposable}. 
A hypergraph $H$ is {\em prime} if it is indecomposable with $v(H)\geq 3$. 

To state Gallai's modular decomposition theorem, we need to define the quotient of a hypergraph by a modular partition (see Section~\ref{section_tournament}). 

\begin{defi}\label{quotient_hyper}
Let $H$ be a hypergraph. 
A partition $P$ of $V(H)$ is a {\em modular partition} of $H$ if $P\subseteq\mathscr{M}(H)$. 
Given a modular partition $P$ of $H$, 
the {\em quotient} $H/P$ of $H$ by $P$ is defined on $V(H/P)=P$ as follows. 
For $\mathcal{E}\subseteq P$, $\mathcal{E}\in E(H/P)$ if $|\mathcal{E}|\geq 2$, and there exists $e\in E(H)$ such that $\mathcal{E}=\{X\in P:X\cap e\neq\emptyset\}$. 
\end{defi}

As for tournaments, we introduce the following strengthening of the notion of a module. 
Let $H$ be a hypergraph. 
A module $M$ of $H$ is {\em strong} if for every module $N$ of $H$, we have 
$$\text{if $M\cap N\neq\emptyset$, then $M\subseteq N$ or $N\subseteq M$.}$$

\begin{nota}\label{nota_Pi_hyper}
We denote by $\Pi(H)$ the set of proper strong modules of $H$ that are maximal under inclusion. 
Clearly, $\Pi(H)$ is a modular partition of $H$ when $v(H)\geq 2$. 
\end{nota}

Gallai's modular decomposition theorem for hypergraphs follows. It is the analogue of Theorem~\ref{Th Gallai}. 

\begin{thm}\label{Thbis_Gallai}
Given a hypergraph $H$ with $v(H)\geq 2$, $H/\Pi(H)$ is an empty hypergraph, a prime hypergraph or a complete graph (i.e. $E(H/\Pi(H))=\binom{\Pi(H)}{2}$). 
\end{thm}

A realization of a 3-uniform hypergraph is defined as follows. 
To begin, we associate with each tournament a 3-uniform hypergraph in the following way. 

\begin{defi}\label{C_3}
The {\em $3$-cycle} is the tournament $C_3=(\{0,1,2\},\{01,12,20\})$. 
Given a tournament $T$, the {\em $C_3$-structure} of $T$ is the 3-uniform hypergraph $C_3(T)$ defined on 
$V(C_3(T))=V(T)$ by 
$$E(C_3(T))=\{X\subseteq V(T):T[X]\ \text{is isomorphic to}\ C_3\}\ \text{(see \cite{BILT04})}.$$
\end{defi}

\begin{defi}
Given a 3-uniform hypergraph $H$, a tournament $T$, with $V(T)=V(H)$, {\em realizes} $H$ if $H=C_3(T)$. 
We say also that $T$ is a {\em realization} of $H$. 
\end{defi}

Whereas a realizable 3-uniform hypergraph and its realizations do not have the same modules, they share the same strong modules. 

\begin{thm}\label{same_strong_modules}
Consider a realizable 3-uniform hypergraph $H$. 
Given a realization $T$ of $H$, $H$ and $T$ share the same strong modules. 
\end{thm}

The next result follows from Theorems~\ref{Thbis_Gallai} and \ref{same_strong_modules}. 

\begin{thm}\label{same_primality}
Consider a realizable 3-uniform hypergraph $H$. 
For a realization $T$ of $H$, we have $H$ is prime if and only if $T$ is prime.
\end{thm}

Lastly, we characterize the realizable 3-uniform hypergraphs. 
To begin, we establish the following theorem by using the modular decomposition tree. 

\begin{thm}\label{realiza_dec}
Given a 3-uniform hypergraph $H$, $H$ is realizable if and only if for every $W\subseteq V(H)$ such that $H[W]$ is prime, $H[W]$ is realizable. 
\end{thm}

We conclude by characterizing the prime and 3-uniform hypergraphs that are realizable (see Theorems~\ref{realiza_critical} and \ref{realiza_non_critical}).

\section{Background on tournaments}\label{section_tournament}

A tournament is a {\em linear order} if it does not contain $C_3$ as a subtournament. 
Given $n\geq 2$, the usual linear order on $\{0,\ldots,n-1\}$ is the tournament 
$L_n=(\{0,\ldots,n-1\},\{pq:0\leq p<q\leq n-1\})$. 
With each tournament $T$, associate its {\em dual} $T^\star$ defined on $V(T^\star)=V(T)$ by 
$A(T^\star)=\{vw:wv\in A(T)\}$.

\subsection{Modular decomposition of tournaments}\label{subs Modular decomposition of tournaments}

Let $T$ be a tournament. 
A subset $M$ of $V(T)$ is a {\em module} \cite{S92} of $T$ provided that for any $x,y\in M$ and $v\in V(T)$, 
if $xv,vy\in A(T)$, then $v\in M$. 
Note that the notions of a module and of an interval coincide for linear orders. 

\begin{nota}\label{module_tournaments}
Given a tournament $T$, the set of the modules of $T$ is denoted by $\mathscr{M}(T)$. 
\end{nota}

We study the set of the modules of a tournament. 
We need the following weakening of the notion of a partitive family. 
Given a set $S$, 
a family $\mathscr{F}$ of subsets of $S$ is a {\em weakly partitive family} on $S$ if it satisfies the following assertions.  
\begin{itemize}
\item $\emptyset\in\mathscr{F}$, $S\in\mathscr{F}$, and for every $x\in S$, 
$\{x\}\in\mathscr{F}$.
\item For any $M,N\in\mathscr{F}$, $M\cap N\in\mathscr{F}$.
\item For any $M,N\in\mathscr{F}$, if $M\cap N\neq\emptyset$, then 
$M\cup N\in\mathscr{F}$.
\item For any $M,N\in\mathscr{F}$, if $M\setminus N\neq\emptyset$, then 
$N\setminus M\in\mathscr{F}$.
\end{itemize}

The set of the modules of a tournament is a weakly partitive family (for instance, see~\cite{ER90b}). 
We generalize the notion of a weakly partitive family as follows. 

\begin{defi}\label{weak_covering}
Let $S$ be a set. 
A {\em weak modular covering} of $S$ is a function $\mathfrak{M}$ which associates with each 
$W\subseteq S$ a set $\mathfrak{M}(W)$ of subsets of $W$, and which satisfies Assertions (A2),...,(A5) (see Definition~\ref{covering}), and the following assertion. 
For each $W\subseteq S$, $\mathfrak{M}(W)$ is a weakly partitive family on $W$. 
\end{defi}

Since the proof of the next proposition is easy and long, we omit it. 

\begin{prop}\label{properties_modules_tournaments}
Given a tournament $T$, the function defined on $2^{V(H)}$, which maps each $W\subseteq V(H)$ to 
$\mathscr{M}(T[W])$, 
is a weak modular covering of $V(T)$. 
\end{prop}

Let $T$ be a tournament. 
By Proposition~\ref{properties_modules_tournaments}, $\emptyset$, $V(T)$ and $\{v\}$, where $v\in V(T)$, are modules of $T$, called {\em trivial}. 
A tournament is {\em indecomposable} if all its modules are trivial, otherwise it is {\em decomposable}. 
A tournament $T$ is {\em prime} if it is indecomposable with $v(T)\geq 3$.  

We define the quotient of a tournament by considering a partition of its vertex set in modules. 
Precisely, let $T$ be a tournament. 
A partition $P$ of $V(T)$ is a \emph{modular partition} of $T$ if $P\subseteq\mathscr{M}(T)$.  
With each modular partition $P$ of $T$, associate the {\em quotient} $T/P$ of $T$ by $P$ defined on 
$V(T/P)=P$ as follows. 
Given $X,Y\in P$ such that $X\neq Y$, $XY\in A(T/P)$ if $xy\in A(T)$, where $x\in X$ and $y\in Y$. 

We need the following strengthening of the
notion of module to obtain an uniform decomposition theorem. 
Given a tournament $T$,
a subset $X$ of $V(T)$ is a {\em strong module} 
\cite{G,MP} of $T$ 
provided that $X$ is a module of $T$ and for every module $M$ of
$T$, if $X\cap M\neq\emptyset $, then $X\subseteq M$ or 
$M\subseteq X$.
With each tournament $T$, with $v(T)\geq 2$, associate 
the set $\Pi(T)$ of the maximal strong module of $T$ under the
inclusion amongst all the proper and strong modules of $T$.
Gallai's modular decomposition theorem follows. 

\begin{thm}[Gallai \cite{G,MP}]\label{Th Gallai} 
Given a tournament $T$ such that $v(T)\geq 2$, $\Pi(T)$ is a modular partition of $T$, and $T/\Pi(T)$ is a linear order or a prime tournament. 
\end{thm}

Theorem~\ref{Th Gallai} is deduced from the following two results. 
We use the following notation. 

\begin{nota}\label{W/P}
Let $P$ be a partition of a set $S$. 
For $W\subseteq S$, $W/P$ denotes the subset $\{X\in P:X\cap W\neq\emptyset\}$ of $P$. 
For $Q\subseteq P$, set 
\begin{equation*}
\cup Q=\bigcup_{X\in Q}X.
\end{equation*}
\end{nota}

\begin{prop}\label{prop_strong_quotient_tournaments}
Given a modular partition $P$ of a tournament $T$, the following two assertions hold. 
\begin{enumerate}
\item If $M$ is a strong module of $T$, then $M/P$ is a strong module of $T/P$. 
\item Suppose that all the elements of $P$ are strong modules of $T$. 
If $\mathcal{M}$ is a strong module of $T/P$, then $\cup\mathcal{M}$ is a strong module of $T$. 
\end{enumerate}
\end{prop}

\begin{thm}\label{tournament_ThB_Gallai} 
Given a tournament $T$, all the strong modules of $T$ are trivial if and only if $T$ is a linear order or a prime tournament. 
\end{thm}

\begin{defi}\label{tree_tournaments}
Given a tournament $T$, the set of the nonempty strong modules of $T$ is denoted by $\mathscr{D}(T)$. 
Clearly, $\mathscr{D}(T)$ endowed with inclusion is a tree called the {\em modular decomposition tree} of $T$. 
\end{defi}

Let $T$ be a tournament. 
The next proposition allows us to obtain all the elements of $\mathscr{D}(T)$ by using successively Theorem~\ref{Th Gallai} from $V(T)$ to the singletons.

\begin{prop}[Ehrenfeucht et al.~\cite{EHR99}]\label{prop_strong_tournaments}
Given a tournament $T$, consider a strong module $M$ of $T$. 
For every $N\subseteq M$, the following two assertions are equivalent
\begin{enumerate}
\item $N$ is a strong module of $T$;
\item $N$ is a strong module of $T[M]$.
\end{enumerate}
\end{prop}

We use the analogue of Proposition~\ref{prop_strong_tournaments} for hypergraphs (see Proposition~\ref{prop_strong_hypergraphs}) to prove Proposition~\ref{P1_real_dec}.

\subsection{Critical tournaments}

\begin{defi}\label{nota_critical}
Given a prime tournament $T$, a vertex $v$ of $T$ is {\em critical} if $T-v$ is decomposable. 
A prime tournament is {\em critical} if all its vertices are critical. 
\end{defi}

Schmerl and Trotter \cite{ST93} characterized the
critical tournaments. 
They obtained the tournaments $T_{2n+1}$, $U_{2n+1}$ and $W_{2n+1}$
defined on $\{0,\ldots, 2n\}$, where $n\geq 1$, as follows. 
\begin{itemize}
\item The tournament $T_{2n+1}$ is obtained from $L_{2n+1}$ by reversing all the arcs between even and odd vertices (see Figure~\ref{T2n+1}). 
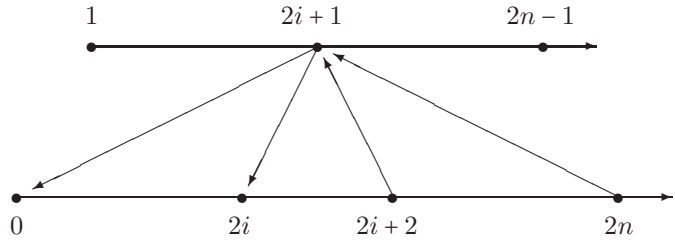
\begin{figure}[h]
\begin{center}
\setlength{\unitlength}{1cm}
\begin{picture}(10,3.5)

\put(1,0.5){$\bullet$}
\put(1,0.1){0}
\put(1.1,0.6){\vector(1,0){8.7}}

\put(4,0.5){$\bullet$}
\put(3.9,0.1){$2i$}

\put(6,0.5){$\bullet$}
\put(5.6,0.1){$2i+2$}
\put(6.1,0.6){\vector(-1,2){0.93}}

\put(9,0.5){$\bullet$}
\put(8.9,0.1){$2n$}
\put(9.1,0.6){\vector(-2,1){3.8}}

\put(2,2.5){$\bullet$}
\put(2,2.9){1}
\put(2.1,2.6){\vector(1,0){6.7}}

\put(5,2.5){$\bullet$}
\put(4.6,2.9){$2i+1$}
\put(5.1,2.6){\vector(-2,-1){3.8}}
\put(5.1,2.6){\vector(-1,-2){0.93}}

\put(8,2.5){$\bullet$}
\put(7.6,2.9){$2n-1$}

\end{picture}
\end{center}
\caption{The tournament $T_{2n+1}$.}
\label{T2n+1}
\end{figure} 
\item The tournament $U_{2n+1}$ is obtained from $L_{2n+1}$ by reversing all the arcs between even vertices (see Figure~\ref{U2n+1}). 
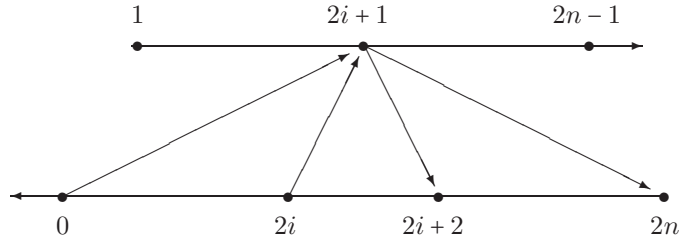
\begin{figure}[h]
\begin{center}
\setlength{\unitlength}{1cm}
\begin{picture}(10,3.5)

\put(1,0.5){$\bullet$}
\put(1,0.1){0}
\put(1.1,0.6){\vector(2,1){3.8}}

\put(4,0.5){$\bullet$}
\put(3.9,0.1){$2i$}
\put(4.1,0.6){\vector(1,2){0.93}}

\put(6,0.5){$\bullet$}
\put(5.6,0.1){$2i+2$}

\put(9,0.5){$\bullet$}
\put(8.9,0.1){$2n$}
\put(9.1,0.6){\vector(-1,0){8.7}}

\put(2,2.5){$\bullet$}
\put(2,2.9){1}
\put(2,2.6){\vector(1,0){6.8}}

\put(5,2.5){$\bullet$}
\put(4.6,2.9){$2i+1$}
\put(5.1,2.6){\vector(1,-2){0.93}}
\put(5.1,2.6){\vector(2,-1){3.8}}

\put(8,2.5){$\bullet$}
\put(7.6,2.9){$2n-1$}

\end{picture}
\end{center}
\caption{The tournament $U_{2n+1}$.}
\label{U2n+1}
\end{figure}
\item The tournament $W_{2n+1}$ is obtained from $L_{2n+1}$ by reversing all the arcs between
$2n$ and the even elements of $\{0,\ldots,2n-1\}$ (see Figure~\ref{W2n+1}). 
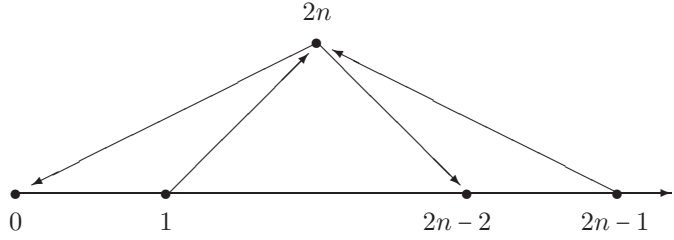
\begin{figure}[h]
\begin{center}
\setlength{\unitlength}{1cm}
\begin{picture}(10,3.5)

\put(1,0.5){$\bullet$}
\put(1,0.1){0}
\put(1.1,0.6){\vector(1,0){8.7}}

\put(3,0.5){$\bullet$}
\put(3,0.1){1}
\put(3.13,0.6){\vector(1,1){1.85}}

\put(7,0.5){$\bullet$}
\put(6.5,0.1){$2n-2$}

\put(9,0.5){$\bullet$}
\put(8.6,0.1){$2n-1$}
\put(9.1,0.6){\vector(-2,1){3.8}}

\put(5,2.5){$\bullet$}
\put(4.9,2.9){$2n$}
\put(5.1,2.6){\vector(-2,-1){3.8}}
\put(5.1,2.6){\vector(1,-1){1.9}}

\end{picture}
\end{center}
\caption{The tournament $W_{2n+1}$.}
\label{W2n+1}
\end{figure}
\end{itemize}

\begin{thm}[Schmerl and Trotter \cite{ST93}]\label{th1 sch trott}
Given a tournament $\tau$, with $v(\tau)\geq 5$, 
$\tau$ is critical if and only if $v(\tau)$ is odd, and $\tau$ is isomorphic to 
$T_{v(\tau)}$, $U_{v(\tau)}$ or $W_{v(\tau)}$. 
\end{thm}

\subsection{The $C_3$-structure of a tournament}

The $C_3$-structure of a tournament (see Definition~\ref{C_3}) is clearly a 3-uniform hypergraph. 
The main theorem of \cite{BILT04} follows. 
It plays an important role in Section~\ref{real}. 

\begin{thm}[Boussa\"{\i}ri et al.~\cite{BILT04}]\label{thm_bilt}
Let $T$ be a prime tournament. 
For every tournament $T'$, if $C_3(T')=C_3(T)$, then $T'=T$ or $T^\star$. 
\end{thm}

\section{Modular decomposition of hypergraphs}

Definition~\ref{defi_module_hyper} is not the usual definition of a module of a hypergraph. 
The usual definition follows. 

\begin{defi}\label{defi_module_usual_hyper}
Let $H$ be a hypergraph. 
A subset $M$ of $V(H)$ is a {\em module} \cite[Definition~2.4]{BDV99} of $H$ if for any 
$e,f\subseteq V(H)$ such that $|e|=|f|$, 
$e\setminus M=f\setminus M$, and $e\setminus M\neq\emptyset$, we have 
$e\in E(H)$ if and only if $f\in E(H)$.  
\end{defi}

\begin{rem}\label{R_not_share}
Given a hypergraph $H$, a module of $H$ in the sense of Definition~\ref{defi_module_hyper} is a module in the sense of 
Definition~\ref{defi_module_usual_hyper}. 
The converse is not true. 
Given $n\geq 3$, consider the 3-uniform hypergraph $H$ defined by 
$V(H)=\{0,\ldots,n-1\}$ and $E(H)=\{01p:2\leq p\leq n-1\}$. 
In the sense of Definition~\ref{defi_module_usual_hyper}, $\{0,1\}$ is a module of $H$ whereas it is not a module of $H$ in the sense of Definition~\ref{defi_module_hyper}. 
\end{rem}

Let $H$ be a realizable 3-uniform hypergraph. 
Consider a realization $T$ of $H$. 
Given $e\in E(H)$, all the modules of $T[e]$ are trivial. 
To handle close modular decompositions for $H$ and $T$, we try to find a definition of a module of $H$ for which all the modules of $H[e]$ are trivial as well. 
This is the case with Definition~\ref{defi_module_hyper}, and not with Definition~\ref{defi_module_usual_hyper}. 
Moreover, note that, with Definition~\ref{defi_module_usual_hyper}, $H$ and $T$ do not share the same strong modules, but but they do with Definition~\ref{defi_module_hyper} (see Theorem~\ref{same_strong_modules}). 
Indeed, consider the 3-uniform hypergraph $H$ defined on $\{0,\ldots,n-1\}$ in Remark~\ref{R_not_share}. 
In the sense of Definition~\ref{defi_module_usual_hyper}, $\{0,1\}$ is a strong module of $H$. 
Now, consider the tournament $T$ obtained from $L_n$ by reversing all the arcs between $0$ and $p\in\{2,\ldots,n-1\}$. 
Clearly, $T$ realizes $H$. 
Since $T[\{0,1,2\}]$ is a 3-cycle, $\{0,1\}$ is not a module of $T$, so it is not a strong module.

\subsection{Modular covering}\label{subs Modular covering}

The purpose of the subsection is to establish Proposition~\ref{properties_modules}. 
To begin, we show that the set of the modules of a hypergraph is a partitive family (see Proposition~\ref{prop_Delta}). 
We need the next three lemmas.

\begin{lem}\label{lem_intersection}
Let $H$ be a hypergraph. 
For any $M,N\in\mathscr{M}(H)$, we have 
$M\cap N\in\mathscr{M}(H)$. 
\end{lem}

\begin{proof}
Consider $M,N\in\mathscr{M}(H)$. 
To show that $M\cap N\in\mathscr{M}(H)$, consider 
$e\in E(H)$ such that $e\cap(M\cap N)\neq\emptyset$ and 
$e\setminus(M\cap N)\neq\emptyset$. 
Since $e\setminus(M\cap N)\neq\emptyset$, assume for instance that 
$e\setminus M\neq\emptyset$. 
Since $M$ is a module of $H$ and $e\cap M\neq\emptyset$, there exists $m\in M$ such that $e\cap M=\{m\}$. 
Since $e\cap(M\cap N)\neq\emptyset$, we obtain $$e\cap(M\cap N)=\{m\}.$$
Let $n\in M\cap N$. 
Since $M$ is a module of $H$, $(e\setminus\{m\})\cup\{n\}\in E(H)$. 
\end{proof}

\begin{lem}\label{lem_union}
Let $H$ be a hypergraph. 
For any $M,N\in\mathscr{M}(H)$, if $M\cap N\neq\emptyset$, then 
$M\cup N\in\mathscr{M}(H)$. 
\end{lem}

\begin{proof}
Consider $M,N\in\mathscr{M}(H)$ such that $M\cap N\neq\emptyset$. 
To show that $M\cup N\in\mathscr{M}(H)$, 
consider 
$e\in E(H)$ such that $e\cap(M\cup N)\neq\emptyset$ and 
$e\setminus(M\cup N)\neq\emptyset$. 
Since $e\cap(M\cup N)\neq\emptyset$, assume for instance that 
$e\cap M\neq\emptyset$. 
Clearly $e\setminus M\neq\emptyset$ because $e\setminus(M\cup N)\neq\emptyset$. 
Since $M$ is a module of $H$, there exists $m\in M$ such that $e\cap M=\{m\}$, and 
\begin{equation}\label{E1_lem_union}
\text{$(e\setminus\{m\})\cup\{n\}\in E(H)$ for every $n\in M$.}
\end{equation} 
Consider $n\in M\cap N$. 
By \eqref{E1_lem_union}, $(e\setminus\{m\})\cup\{n\}\in E(H)$. 
Set $$f=(e\setminus\{m\})\cup\{n\}.$$
Clearly $n\in f\cap N$. 
Furthermore, consider $p\in e\setminus(M\cup N)$. 
Since $m\in M$, we have $p\neq m$, and hence $p\in f\setminus N$. 
Since $N$ is a module of $H$, we obtain $f\cap N=\{n\}$ and 
\begin{equation}\label{E2_lem_union}
\text{$(f\setminus\{n\})\cup\{n'\}\in E(H)$ for every $n'\in N$.}
\end{equation}
Since $(f\setminus\{n\})\cup\{n'\}=(e\setminus\{m\})\cup\{n'\}$ for every $n'\in N$, it follows from \eqref{E2_lem_union} that 
\begin{equation}\label{E3_lem_union}
\text{$(e\setminus\{m\})\cup\{n'\}\in E(H)$ for every $n'\in N$.}
\end{equation}
Therefore, it follows from \eqref{E1_lem_union} and \eqref{E3_lem_union} that 
\begin{equation*}
\text{$(e\setminus\{m\})\cup\{n\}\in E(H)$ for every $n\in M\cup N$.}
\end{equation*}
Moreover, since $f\cap N=\{n\}$, we have 
\begin{align*}
e\cap N=&\ (\{m\}\cup(e\setminus\{m\}))\cap N\\
=&\ (\{m\}\cup(f\setminus\{n\}))\cap N\\
=&\ \{m\}\cap N,
\end{align*} 
and hence $e\cap N\subseteq\{m\}$. 
Since $e\cap M=\{m\}$, we obtain $e\cap(M\cup N)=\{m\}$. 
Consequently, $M\cup N$ is a module of $H$. 
\end{proof}

\begin{lem}\label{lem_moins}
Let $H$ be a hypergraph. 
For any $M,N\in\mathscr{M}(H)$, if $M\setminus N\neq\emptyset$, then 
$N\setminus M\in\mathscr{M}(H)$. 
\end{lem}

\begin{proof}
Consider $M,N\in\mathscr{M}(H)$ such that $M\setminus N\neq\emptyset$. 
To show that $N\setminus M\in\mathscr{M}(H)$, 
consider 
$e\in E(H)$ such that $e\cap(N\setminus M)\neq\emptyset$ and 
$e\setminus(N\setminus M)\neq\emptyset$. 
We distinguish the following two cases. 
\begin{enumerate}
\item Suppose that $e\setminus N\neq\emptyset$. 
Since $N$ is a module of $H$ and $e\cap N\neq\emptyset$, there exists $n\in N$ such that 
$e\cap N=\{n\}$, and 
$(e\setminus\{n\})\cup\{n'\}\in E(H)$ for every $n'\in N$. 
Since $e\cap N=\{n\}$ and $e\cap(N\setminus M)\neq\emptyset$, we obtain 
$e\cap(N\setminus M)=\{n\}$. 
Therefore, 
\begin{equation*}
\begin{cases}
e\cap(N\setminus M)=\{n\}\\
\text{and}\\
\text{$(e\setminus\{n\})\cup\{n'\}\in E(H)$ for every $n'\in N\setminus M$. }
\end{cases}
\end{equation*}
\item Suppose that $e\subseteq N$. 
Since $e\setminus(N\setminus M)\neq\emptyset$, we have $e\cap(M\cap N)\neq\emptyset$. 
Thus $e\cap M\neq\emptyset$, and $e\setminus M\neq\emptyset$ because 
$e\cap(N\setminus M)\neq\emptyset$. 
Since $M$ is a module of $H$, there exists $m\in M$ such that 
$e\cap M=\{m\}$, and 
\begin{equation}\label{E1_lem_moins}
\text{$(e\setminus\{m\})\cup\{m'\}\in E(H)$ for every $m'\in M$. }
\end{equation}
Since $e\cap(M\cap N)\neq\emptyset$, $m\in M\cap N$. 
Consider $p\in M\setminus N$ and $q\in e\cap(N\setminus M)$. 
Set $$f=(e\setminus\{m\})\cup\{p\}.$$
By \eqref{E1_lem_moins}, $f\in E(H)$. 
Clearly, $p\in f\setminus N$ and $q\in f\cap N$. 
Since $N$ is a module of $H$, we have 
$f\cap N=\{q\}$, and 
\begin{equation}\label{E3_lem_moins}
\text{$(f\setminus\{q\})\cup\{r\}\in E(H)$ for every $r\in N\setminus M$. }
\end{equation}
Since $f\cap N=\{q\}$, we obtain $e=mq$, and hence 
\begin{equation}\label{E4_lem_moins}
e\cap(N\setminus M)=\{q\}. 
\end{equation}
Since $e=mq$, we get $f=pq$. 
Moreover, for each $r\in N\setminus M$, set $$g_r=(f\setminus\{q\})\cup\{r\}.$$
Since $f=pq$, we have $g_r=pr$. 
By \eqref{E3_lem_moins}, $g_r\in E(H)$. 
Clearly, $p\in g_r\cap M$ and $r\in g_r\setminus M$. 
Since $M$ is a module of $H$, we obtain $(g_r\setminus\{p\})\cup\{m\}\in E(H)$. 
Since $g_r=pr$, we have 
$$(g_r\setminus\{p\})\cup\{m\}=mr=(e\setminus\{q\})\cup\{r\}$$ because $e=mq$. 
Consequently, for each $r\in N\setminus M$, $(e\setminus\{q\})\cup\{r\}\in E(H)$, where 
$\{q\}=e\cap(N\setminus M)$ by \eqref{E4_lem_moins}. \qedhere
\end{enumerate}
\end{proof}

\begin{prop}\label{prop_Delta}
Given a hypergraph $H$, $\mathscr{M}(H)$ is a partitive family on $V(H)$. 
\end{prop}

\begin{proof}
It is easy to verify that $\emptyset\in\mathscr{M}(H)$, $V(H)\in\mathscr{M}(H)$, and for every $v\in V(H)$, 
$\{v\}\in\mathscr{M}(H)$. 
Therefore, it follows from Lemmas \ref{lem_intersection}, \ref{lem_union} and 
\ref{lem_moins} that $\mathscr{M}(H)$ is a weakly partitive family on $V(H)$. 
To prove that $\mathscr{M}(H)$ is a partitive family on $V(H)$, 
consider any $M,N\in\mathscr{M}(H)$ such that $M\setminus N\neq\emptyset$, 
$N\setminus M\neq\emptyset$ and $M\cap N\neq\emptyset$. 
We have to show that $(M\setminus N)\cup(N\setminus M)\in\mathscr{M}(H)$. 
Hence consider $e\in E(H)$ such that 
$e\cap((M\setminus N)\cup(N\setminus M))\neq\emptyset$ and 
$e\setminus((M\setminus N)\cup(N\setminus M))\neq\emptyset$. 
Since $e\cap((M\setminus N)\cup(N\setminus M))\neq\emptyset$, assume for instance that $e\cap(M\setminus N)\neq\emptyset$. 
Clearly $e\setminus(M\setminus N)\neq\emptyset$ because 
$e\setminus((M\setminus N)\cup(N\setminus M))\neq\emptyset$. 
Since $N\setminus M\neq\emptyset$, it follows from Lemma~\ref{lem_moins} that 
$M\setminus N$ is a module of $H$. 
Thus, there exists $m\in M\setminus N$ such that $e\cap(M\setminus N)=\{m\}$. 
We distinguish the following two cases. 
\begin{enumerate}
\item Suppose that $e\subseteq M$. 
Since $e\setminus((M\setminus N)\cup(N\setminus M))\neq\emptyset$, 
$e\cap(M\cap N)\neq\emptyset$. 
Therefore $e\cap N\neq\emptyset$. 
Furthermore, since $e\cap(M\setminus N)\neq\emptyset$, we have 
$e\setminus N\neq\emptyset$. 
Since $N$ is a module of $H$, there exists $n\in N$ such that $e\cap N=\{n\}$. 
Since $e\cap(M\cap N)\neq\emptyset$, we get $e\cap(M\cap N)=\{n\}$. 
Since $e\subseteq M$ and $e\cap(M\setminus N)=\{m\}$, we obtain $e=mn$. 
It follows that 
\begin{equation}\label{prop_Delta_E1}
e\cap((M\setminus N)\cup(N\setminus M))=\{m\}.
\end{equation}
Let $p\in(M\setminus N)\cup(N\setminus M)$. 
We have to show that 
\begin{equation}\label{prop_Delta_E2}
(e\setminus\{m\})\cup\{p\}=np\in E(H).
\end{equation}
Recall that $M\setminus N$ is a module of $H$. 
Consequently \eqref{prop_Delta_E2} holds whenever $p\in M\setminus N$. 
Suppose that $p\in N\setminus M$. 
Since $N$ is a module of $H$ and $mn\in E(H)$, we get $mp\in E(H)$. 
Now, since $M$ is a module of $H$ and $mp\in E(H)$, we obtain $np\in E(H)$. 
It follows that \eqref{prop_Delta_E2} holds for each 
$p\in(M\setminus N)\cup(N\setminus M)$. 
Lastly, it follows from \eqref{prop_Delta_E1} that 
there exists $m\in M\setminus N$ such that 
\begin{equation*}
\begin{cases}
e\cap((M\setminus N)\cup(N\setminus M))=\{m\}\\
\text{and}\\
\text{for each $p\in(M\setminus N)\cup(N\setminus M)$, 
$(e\setminus\{m\})\cup\{p\}\in E(H)$}.
\end{cases}
\end{equation*}
\item Suppose that $e\setminus M\neq\emptyset$. 
Since $e\cap (M\setminus N)=\{m\}$, $m\in e\cap M$. 
Since $M$ is a module of $H$, there exists $m'\in M$ such that $e\cap M=\{m'\}$. 
Since $e\cap (M\setminus N)=\{m\}$, we have $m=m'$, and hence 
$$e\cap (M\setminus N)=e\cap M=\{m\}.$$
It follows that $e\cap(M\cap N)=\emptyset$. 
Since $e\setminus((M\setminus N)\cup(N\setminus M))\neq\emptyset$, we obtain 
$$e\setminus(M\cup N)\neq\emptyset.$$
Since $M\cap N\neq\emptyset$, it follows from Lemma~\ref{lem_union} that $M\cup N$ is a module of $H$. 
Therefore, there exists $p\in M\cup N$ such that $e\cap(M\cup N)=\{p\}$, and for every $q\in M\cup N$, 
$(e\setminus\{p\})\cup\{q\}\in E(H)$. 
Since $e\cap M=\{m\}$, we get $p=m$. 
Thus, $e\cap(M\cup N)=\{m\}$, and hence 
\begin{equation}\label{prop_Delta_E4}
e\cap((M\setminus N)\cup(N\setminus M))=\{m\}.
\end{equation}
Since $p=m$, we have $(e\setminus\{m\})\cup\{q\}\in E(H)$ for every $q\in M\cup N$. 
It follows that 
\begin{equation*}
\text{$(e\setminus\{m\})\cup\{q\}\in E(H)$ for every $q\in(M\setminus N)\cup(N\setminus M)$,}
\end{equation*}
where $\{m\}=e\cap((M\setminus N)\cup(N\setminus M))$ by \eqref{prop_Delta_E4}. \qedhere
\end{enumerate}
\end{proof}

To prove Proposition~\ref{properties_modules}, we need the next four lemmas. 

\begin{lem}\label{lem_(A2)}
Given a hypergraph $H$, consider subsets $W$ and $W'$ of $V(H)$. 
If $W\subseteq W'$, then $\{M'\cap W:M'\in\mathscr{M}(H[W'])\}\subseteq\mathscr{M}(H[W])$ 
(see Definition~\ref{covering}, Assertion~(A2)). 
\end{lem}

\begin{proof}
Let $M'$ be a module of $H[W']$. 
To show that $M'\cap W$ is a module of $H[W]$, consider $e\in E(H[W])$ such that 
$e\cap(M'\cap W)\neq\emptyset$ and $e\setminus(M'\cap W)\neq\emptyset$. 
We obtain $e\in E(H[W'])$ and $e\cap M'\neq\emptyset$. 
Since $e\setminus(M'\cap W)\neq\emptyset$ and $e\subseteq W$, we get 
$e\setminus M'\neq\emptyset$. 
Since $M'$ is a module of $H[W']$, there exists $m'\in M'$ such that 
$e\cap M'=\{m'\}$, and $(e\setminus\{m'\})\cup\{n'\}\in E(H[W'])$ for each 
$n'\in M'$. 
Let $n'\in M'\cap W$. 
Since $e\subseteq W$, $(e\setminus\{m'\})\cup\{n'\}\subseteq W$. 
Hence $(e\setminus\{m'\})\cup\{n'\}\in E(H[W])$ because 
$(e\setminus\{m'\})\cup\{n'\}\in E(H[W'])$. 
Moreover, since $e\cap(M'\cap W)\neq\emptyset$ and $e\cap M'=\{m'\}$, we obtain 
$e\cap(M'\cap W)=\{m'\}$. 
\end{proof}

\begin{lem}\label{lem_(A3)}
Given a hypergraph $H$, consider subsets $W$ and $W'$ of $V(H)$ such that $W\subseteq W'$. 
If $W\in\mathscr{M}(H[W'])$, then $\{M'\in\mathscr{M}(H[W']):M'\subseteq W\}=\mathscr{M}(H[W])$ 
(see Definition~\ref{covering}, Assertion~(A3)). 
\end{lem}

\begin{proof}
By Lemma~\ref{lem_(A2)}, 
$\{M'\in\mathscr{M}(H[W']):M'\subseteq W\}\subseteq\mathscr{M}(H[W])$. 
Conversely, consider a module $M$ of $H[W]$. 
To prove that $M$ is a module of $H[W']$, consider $e\in E(H[W'])$ such that 
$e\cap M\neq\emptyset$ and $e\setminus M\neq\emptyset$. 
We distinguish the following two cases. 
\begin{enumerate}
\item Suppose that $e\subseteq W$. 
We obtain $e\in E(H[W])$. 
Since $M$ is a module of $H[W]$, there exists $m\in M$ such that $e\cap M=\{m\}$, and for each $n\in M$, we have $(e\setminus\{m\})\cup\{n\}\in E(H[W])$. 
Hence $(e\setminus\{m\})\cup\{n\}\in E(H[W'])$. 
\item Suppose that $e\setminus W\neq\emptyset$. 
Clearly, $e\cap W\neq\emptyset$ because $e\cap M\neq\emptyset$. 
Since $W$ is a module of $H[W']$, there exists $w\in W$ such that $e\cap W=\{w\}$. 
Furthermore, 
\begin{equation}\label{properties_modules_E1}
\text{for each $w'\in W$, $(e\setminus\{w\})\cup\{w'\}\in E(H[W'])$. }
\end{equation}
Since $e\cap M\neq\emptyset$, we get $e\cap M=\{w\}$. 
Clearly, it follows from \eqref{properties_modules_E1} that $(e\setminus\{w\})\cup\{w'\}\in E(H[W'])$ 
for each $w'\in M$. \qedhere
\end{enumerate}
\end{proof}

\begin{lem}\label{lem_(A4)}
Given a hypergraph $H$, consider subsets $W$ and $W'$ of $V(H)$ such that $W\subseteq W'$. 
For any $M\in\mathscr{M}(H[W])$ and $M'\in\mathscr{M}(H[W'])$, 
if $M\cap M'=\emptyset$ and $M'\cap W\neq\emptyset$, then $M\in\mathscr{M}(H[W\cup M'])$ 
(see Definition~\ref{covering}, Assertion~(A4)). 
\end{lem}

\begin{proof}
Consider a module $M$ of $H[W]$ and a module $M'$ of $H[W']$ such that 
$M\cap M'=\emptyset$ and $M'\cap W\neq\emptyset$. 
We have to show that $M$ is a module of $H[W\cup M']$. 
Hence consider $e\in E(H[W\cup M'])$ such that $e\cap M\neq\emptyset$ and $e\setminus M\neq\emptyset$. 
We distinguish the following two cases. 
\begin{enumerate}
\item Suppose that $e\subseteq W$. 
We obtain $e\in E(H[W])$. 
Since $M$ is a module of $H[W]$, there exists $m\in M$ such that $e\cap M=\{m\}$, and for each $n\in M$, we have $(e\setminus\{m\})\cup\{n\}\in E(H[W])$. 
Hence $(e\setminus\{m\})\cup\{n\}\in E(H[W\cup M'])$. 
\item Suppose that $e\setminus W\neq\emptyset$. 
We obtain $e\cap(M'\setminus W)\neq\emptyset$. 
Since $e\cap M\neq\emptyset$, we have $e\setminus M'\neq\emptyset$. 
Since $M'$ is a module of $H[W']$, there exists $m'\in M'$ such that $e\cap M'=\{m'\}$, and 
\begin{equation}\label{properties_modules_E2}
\text{for each $n'\in M'$, $(e\setminus\{m'\})\cup\{n'\}\in E(H[W'])$. }
\end{equation}
Since $e\cap(M'\setminus W)\neq\emptyset$ and $e\cap M'=\{m'\}$, we get $e\cap(M'\setminus W)=\{m'\}$. 
Let $w'\in W\cap M'$. 
Set $$f=(e\setminus\{m'\})\cup\{w'\}.$$
By \eqref{properties_modules_E2}, $f\in E(H[W'])$. 
Furthermore, since $e\cap(M'\setminus W)=\{m'\}$, we obtain $f\subseteq W$, and hence $f\in E(H[W])$. 
Since $e\cap M\neq\emptyset$, we have $f\cap M\neq\emptyset$. 
Moreover, $w'\in f\setminus M$ because $w'\in W\cap M'$ and $M\cap M'=\emptyset$. 
Since $M$ is a module of $H[W]$, there exists $m\in M$ such that $f\cap M=\{m\}$. 
Since $f=(e\setminus\{m'\})\cup\{w'\}$, with $m',w'\not\in M$, we get $e\cap M=f\cap M$, so 
$$e\cap M=\{m\}.$$
Lastly, consider $n\in M$. 
We have to verify that 
\begin{equation}\label{properties_modules_E3}
(e\setminus\{m\})\cup\{n\}\in E(H[W']). 
\end{equation}
Set $$g_n=(f\setminus\{m\})\cup\{n\}.$$ 
Since $M$ is a module of $H[W]$ such that $f\cap M=\{m\}$ and $w'\in f\setminus M$, 
$g_n\in E(H[W])$. 
Hence $g_n\in E(H[W'])$. 
Since $n\in g_n\cap M$ and $M\cap M'=\emptyset$, $n\in g_n\setminus M'$. 
Clearly, $w'\in M'$ because $w'\in W\cap M'$. 
Furthermore, $w'\in f$ because $f=(e\setminus\{m'\})\cup\{w'\}$. 
Since $g_n=(f\setminus\{m\})\cup\{n\}$, $m\in M$ and $M\cap M'=\emptyset$, we have $w'\in g_n$. 
It follows that $w'\in g_n\cap M'$. 
Since $M'$ is a module of $H[W']$, we have $g_n\cap M'=\{w'\}$ and 
$(g_n\setminus\{w'\})\cup\{m'\}\in E(H[W'])$. 
We have 
\begin{align*}
(g_n\setminus\{w'\})\cup\{m'\}&=(((f\setminus\{m\})\cup\{n\})\setminus\{w'\})\cup\{m'\}\\
&=(f\setminus\{m,w'\})\cup\{m',n\}\\
&=(((e\setminus\{m'\})\cup\{w'\})\setminus\{m,w'\})\cup\{m',n\}\\
&=(e\setminus\{m,m',w'\})\cup\{m',n,w'\}\\
&=(e\setminus\{m\})\cup\{n\}.
\end{align*}
Therefore $(e\setminus\{m\})\cup\{n\}\in E(H[W'])$. 
Since $e\subseteq W\cup M'$ and $n\in M$, we get $(e\setminus\{m\})\cup\{n\}\in E(H[W\cup M'])$. Hence \eqref{properties_modules_E3} holds. \qedhere
\end{enumerate}
\end{proof}

\begin{lem}\label{lem_(A5)}
Given a hypergraph $H$, consider subsets $W$ and $W'$ of $V(H)$ such that $W\subseteq W'$. 
For any $M\in\mathscr{M}(H[W])$ and $M'\in\mathscr{M}(H[W'])$, 
if $M\cap M'\neq\emptyset$, then $M\cup M'\in\mathscr{M}(H[W\cup M'])$ 
(see Definition~\ref{covering}, Assertion~(A5)). 
\end{lem}

\begin{proof}
Consider a module $M$ of $H[W]$ and a module $M'$ of $H[W']$ such that $M\cap M'\neq\emptyset$. 
We have to prove that $M\cup M'$ is a module of $H[W\cup M']$. 
Hence consider $e\in E(H[W\cup M'])$ such that $e\cap(M\cup M')\neq\emptyset$ and 
$e\setminus(M\cup M')\neq\emptyset$. 
Let $m\in M\cap M'$. 
We distinguish the following two cases. 
\begin{enumerate}
\item Suppose that $e\cap M'\neq\emptyset$. 
Clearly $e\in E(H[W'])$. 
Moreover, $e\setminus M'\neq\emptyset$ because $e\setminus(M\cup M')\neq\emptyset$. 
Since $M'$ is a module of $H[W']$, there exists $m'\in M'$ such that $e\cap M'=\{m'\}$, and 
$(e\setminus\{m'\})\cup\{n'\}\in E(H[W'])$ for every $n'\in M'$. 
Hence, for every $n'\in M'$, we have 
\begin{equation}\label{properties_modules_E4}
(e\setminus\{m'\})\cup\{n'\}\in E(H[W\cup M']).
\end{equation} 
In particular, $(e\setminus\{m'\})\cup\{m\}\in E(H[W\cup M'])$. 
Set $$f=(e\setminus\{m'\})\cup\{m\}.$$
Since $e\cap M'=\{m'\}$, we obtain $f\cap M'=\{m\}$. 
Hence $m\in f\cap M$. 
It follows that $f\in E(H[W])$ because $e\in E(H[W\cup M'])$. 
Clearly $e\setminus M\neq\emptyset$ because $e\setminus(M\cup M')\neq\emptyset$. 
Since $M$ is a module of $H[W]$, there exists $n\in M$ such that $f\cap M=\{n\}$, and 
$(f\setminus\{n\})\cup\{p\}\in E(H[W])$ for every $p\in M$. 
Since $m\in f\cap M$, we get $m=n$. 
Therefore, $f\cap M=f\cap M'=\{m\}$. 
It follows that $f\cap(M\cup M')=\{m\}$, so $$e\cap(M\cup M')=\{m'\}.$$ 
By \eqref{properties_modules_E4}, it remains to show that 
$(e\setminus\{m'\})\cup\{n\}\in E(H[W\cup M'])$ for each $n\in M$. 
Let $n\in M$. 
Recall that $f\cap(M\cup M')=\{m\}$ and $e\cap(M\cup M')=\{m'\}$. 
Thus $e\setminus(M\cup M')=f\setminus(M\cup M')$. 
Hence $f\setminus(M\cup M')\neq\emptyset$ because $e\setminus(M\cup M')\neq\emptyset$. 
It follows that $f\setminus M\neq\emptyset$. 
Recall that $f\in E(H[W])$. 
Since $M$ is a module of $H[W]$, we obtain $(f\setminus\{m\})\cup\{n\}\in E(H[W])$. 
We have 
\begin{align*}
(f\setminus\{m\})\cup\{n\}&=((e\setminus\{m'\})\cup\{m\})\setminus\{m\})\cup\{n\}\\
&=(e\setminus\{m'\})\cup\{n\}.
\end{align*}
Therefore $(e\setminus\{m'\})\cup\{n\}\in E(H[W])$, so $(e\setminus\{m'\})\cup\{n\}\in E(H[W\cup M'])$. 
\item Suppose that $e\cap M'=\emptyset$. 
We get $e\in E(H[W])$. 
Clearly $e\setminus M\neq\emptyset$ because $e\setminus(M\cup M')\neq\emptyset$. 
Furthermore, since $e\cap(M\cup M')\neq\emptyset$ and $e\cap M'=\emptyset$, we obtain 
$e\cap(M\setminus M')\neq\emptyset$. 
Since $M$ is a module of $H[W]$, there exists $q\in M$ such that 
\begin{equation}\label{properties_modules_E5}
e\cap M=\{q\}
\end{equation} 
and 
\begin{equation}\label{properties_modules_E6}
\text{for every $r\in M$, $(e\setminus\{q\})\cup\{r\}\in E(H[W])$.}
\end{equation} 
Since $e\cap M'=\emptyset$, it follows from \eqref{properties_modules_E5} that $q\in M\setminus M'$ 
and 
\begin{equation}\label{properties_modules_E6.a}
e\cap(M\cup M')=\{q\}.
\end{equation} 
By \eqref{properties_modules_E6}, 
$(e\setminus\{q\})\cup\{m\}\in E(H[W])$. 
Set $$e'=(e\setminus\{q\})\cup\{m\}.$$
Clearly, $m\in e'\cap M'$. 
Moreover, since $e\cap(M\cup M')=\{q\}$, we obtain 
\begin{equation*}
\begin{cases}
e'\cap(M\cup M')=\{m\}\\
\text{and}\\
e\setminus(M\cup M')=e'\setminus(M\cup M').
\end{cases}
\end{equation*} 
Therefore $e'\cap(M\cup M')\neq\emptyset$, and 
$e'\setminus(M\cup M')\neq\emptyset$ because 
$e\setminus(M\cup M')\neq\emptyset$. 
It follows from the first case above applied with $e'$ that 
\begin{equation}\label{properties_modules_E8}
\text{for every $s\in M\cup M'$, $(e'\setminus\{m\})\cup\{s\}\in E(H[W\cup M'])$.}
\end{equation} 
Recall that $e\cap(M\cup M')=\{q\}$ by \eqref{properties_modules_E6.a}. 
Consequently, we have to show that 
$(e\setminus\{q\})\cup\{s\}\in E(H[W\cup M'])$ for every $s\in M\cup M'$. 
Let $s\in M\cup M'$. 
We have 
\begin{align*}
(e'\setminus\{m\})\cup\{s\}&=((e\setminus\{q\})\cup\{m\})\setminus\{m\})\cup\{s\}\\
&=(e\setminus\{q\})\cup\{s\}.
\end{align*}
It follows from \eqref{properties_modules_E8} that 
$(e\setminus\{q\})\cup\{s\}\in E(H[W\cup M'])$. \qedhere
\end{enumerate}
\end{proof}

Now, we can prove Proposition~\ref{properties_modules}. 

\begin{proof}[Proof of Proposition~\ref{properties_modules}]
For Assertion (A1) (see Definition~\ref{covering}), consider $W\subseteq V(H)$. 
By Proposition~\ref{prop_Delta}, 
$\mathscr{M}(H[W])$ is a partitive family on $W$. 
Furthermore, it follows from Lemmas~\ref{lem_(A2)}, \ref{lem_(A3)}, \ref{lem_(A4)}, and \ref{lem_(A5)} that 
Assertions (A2), (A3), (A4), and (A5) hold. 
\end{proof}

\subsection{Gallai's decomposition}\label{subs Gallai's decomposition}

The purpose of the subsection is to demonstrate Theorem~\ref{Thbis_Gallai}. 
We use the following definition.

\begin{defi}\label{defi_transverse}
Let $P$ be a partition of a set $S$. 
Consider $Q\subseteq P$. 
A subset $W$ of $S$ is a {\em transverse} of $Q$ if $W\subseteq\cup Q$ and $|W\cap X|=1$ for each $X\in Q$. 
\end{defi}
 
 The next remark makes clearer Definition~\ref{quotient_hyper}. 

\begin{rem}\label{rem_partition}
Consider a modular partition $P$ of a hypergraph $H$. 
Let $e\in E(H)$ such that $|e/P|\geq 2$ (see Notation~\ref{W/P}).  
Given $X\in e/P$, we have $e\cap X\neq\emptyset$, and $e\setminus X\neq\emptyset$ because $|e/P|\geq 2$. 
Since $X$ is a module of $H$, we obtain $|e\cap X|=1$. 
Therefore, $e$ is a transverse of $e/P$. 
Moreover, since each element of $e/P$ is a module of $H$, we obtain that each transverse of $e/P$ is an edge of $H$. 

Given $\mathcal{E}\subseteq P$ such that $|\mathcal{E}|\geq 2$, 
it follows that $\mathcal{E}\in E(H/P)$ if and only if every transverse of $\mathcal{E}$ is an edge of $H$. 

Lastly, consider a transverse $t$ of $P$. 
The function $\theta_t$ from $t$ to $P$, which maps each $x\in t$ to the unique element of $P$ containing $x$, 
is an isomorphism from $H[t]$ onto $H/P$. 
\end{rem}

In the next proposition, we study the links between the modules of a hypergraph with those of its quotients. 

\begin{prop}\label{prop1_partition}
Given a modular partition $P$ of a hypergraph $H$, the following two assertions hold
\begin{enumerate}
\item if $M$ is a module of $H$, then $M/P$ is a module of $H/P$ (see Notation~\ref{W/P}); 
\item if $\mathcal{M}$ is a module of $H/P$, then $\cup\mathcal{M}$ is a module of $H$. 
\end{enumerate}
\end{prop}

\begin{proof}
For the first assertion, consider a module $M$ of $H$. 
Consider a transverse $t$ of $P$ such that 
\begin{equation}\label{E1_prop1_partition}
\text{for each $X\in M/P$, $t\cap X\subseteq M$.}
\end{equation}
By Lemma~\ref{lem_(A2)}, $M\cap t$ is a module of $H[t]$. 
Since $\theta_t$ is an isomorphism from $H[t]$ onto $H/P$(see Remark~\ref{rem_partition}), 
$$\text{$\theta_t(M\cap t)$, that is, $M/P$}$$ is a module of $H/P$. 

For the second assertion, consider a module $\mathcal{M}$ of $H/P$. 
Let $t$ be any transverse of $P$. 
Since $\theta_t$ is an isomorphism from $H[t]$ onto $H/P$, 
$(\theta_t)^{-1}(\mathcal{M})$ is a module of $H[t]$. 
Set $$\mu=(\theta_t)^{-1}(\mathcal{M}).$$
Denote the elements of $\mathcal{M}$ by $X_0,\ldots,X_m$. 
We verify by induction on $i\in\{0,\ldots,m\}$ that $\mu\cup(X_0\cup\ldots\cup X_i)$ is a module of 
$H[t\cup(X_0\cup\ldots\cup X_i)]$. 
It follows from Lemma~\ref{lem_(A5)} that $\mu\cup X_0$ is a module of 
$H[t\cup X_0]$. 
Given $0\leq i<m$, suppose that $\mu\cup(X_0\cup\ldots\cup X_i)$ is a module of 
$H[t\cup(X_0\cup\ldots\cup X_i)]$. 
Similarly, it follows from Lemma~\ref{lem_(A5)} that $\mu\cup(X_0\cup\ldots\cup X_{i+1})$ is a module of 
$H[t\cup(X_0\cup\ldots\cup X_{i+1})]$. 
By induction, we obtain that $\mu\cup(X_0\cup\ldots\cup X_m)$ is a module of 
$H[t\cup(X_0\cup\ldots\cup X_m)]$. 
Observe that $$\mu\cup(X_0\cup\ldots\cup X_m)=\cup\mathcal{M}.$$
Lastly, denote the elements of $P\setminus\mathcal{M}$ by $Y_0,\ldots,Y_n$. 
Using Lemma~\ref{lem_(A4)}, we show by induction on $0\leq j\leq n$ that 
$(\cup\mathcal{M})$ is a module of 
$H[t\cup(X_0\cup\ldots\cup X_m)\cup(Y_0\cup\ldots\cup Y_j)]$. 
Consequently, we obtain that $(\cup\mathcal{M})$ is a module of 
$H[t\cup(X_0\cup\ldots\cup X_m)\cup(Y_0\cup\ldots\cup Y_n)]$, that is, $H$. 
\end{proof}

The next proposition is similar to Proposition~\ref{prop1_partition}, but it is devoted to strong modules. 
It is the analogue of Proposition~\ref{prop_strong_quotient_tournaments} for hypergraphs. 

\begin{prop}\label{prop2_partition}
Given a modular partition $P$ of a hypergraph $H$, the following two assertions hold. 
\begin{enumerate}
\item If $M$ is a strong module of $H$, then $M/P$ is a strong module of $H/P$ (see Notation~\ref{W/P}). 
\item Suppose that all the elements of $P$ are strong modules of $H$. 
If $\mathcal{M}$ is a strong module of $H/P$, then $\cup\mathcal{M}$ is a strong module of $H$. 
\end{enumerate}
\end{prop}

\begin{proof}
For the first assertion, consider a strong module $M$ of $H$. 
By the first assertion of Proposition~\ref{prop1_partition}, $M/P$ is a module of $H/P$. 
To show that $M/P$ is strong, consider a module $\mathcal{M}$ of $H/P$ such that 
$(M/P)\cap\mathcal{M}\neq\emptyset$. 
By the second assertion of Proposition~\ref{prop1_partition}, $\cup\mathcal{M}$ is a module of $H$. 
Furthermore, since $(M/P)\cap\mathcal{M}\neq\emptyset$, there exists $X\in(M/P)\cap\mathcal{M}$. 
We get $X\cap M\neq\emptyset$ and $X\subseteq\cup\mathcal{M}$. 
Therefore, $M\cap(\cup\mathcal{M})\neq\emptyset$. 
Since $M$ is a strong module of $H$, we obtain $\cup\mathcal{M}\subseteq M$ or 
$M\subseteq\cup\mathcal{M}$. 
In the first instance, we get $\mathcal{M}\subseteq M/P$, and, in the second one, we get 
$M/P\subseteq\mathcal{M}$. 

For the second assertion, suppose that all the elements of $P$ are strong modules of $H$. 
Consider a strong module $\mathcal{M}$ of $H/P$. 
To begin, we make two observations. 
First, if $\mathcal{M}=\emptyset$, then $\cup\mathcal{M}=\emptyset$, and hence $\cup\mathcal{M}$ is a strong module of $H$. 
Second, if $|\mathcal{M}|=1$, then $\cup\mathcal{M}\in P$, and hence $\cup\mathcal{M}$ is a strong module of $H$ because all the elements of $P$ are. 
Now, suppose that 
\begin{equation}\label{E1_prop2_partition}
|\mathcal{M}|\geq 2.
\end{equation}
By the second assertion of Proposition~\ref{prop1_partition}, $\cup\mathcal{M}$ is a module of $H$. 
To show that $\cup\mathcal{M}$ is strong, consider a module $M$ of $H$ such that 
$M\cap(\cup\mathcal{M})\neq\emptyset$. 
Let $x\in M\cap(\cup\mathcal{M})$. 
Denote by $X$ the unique element of $P$ containing $x$. 
We get $X\in (M/P)\cap\mathcal{M}$. 
Since $\mathcal{M}$ is a strong module of $H/P$, we obtain $M/P\subseteq\mathcal{M}$ or 
$\mathcal{M}\subseteq M/P$. 
In the first instance, we obtain $\cup(M/P)\subseteq\cup\mathcal{M}$, so we have 
$M\subseteq\cup(M/P)\subseteq\cup\mathcal{M}$. 
Lastly, suppose $\mathcal{M}\subseteq M/P$. 
It follows from \eqref{E1_prop2_partition} that $$|M/P|\geq 2.$$
Let $Y\in M/P$. 
We have $Y\cap M\neq\emptyset$. 
Since $|M/P|\geq 2$, we have $M\setminus Y\neq\emptyset$. 
Since $Y$ is a strong module of $P$, we obtain $Y\subseteq M$. 
It follows that $M=\cup(M/P)$. 
Since $\mathcal{M}\subseteq M/P$, we obtain $\cup\mathcal{M}\subseteq \cup(M/P)$, and hence 
$\cup\mathcal{M}\subseteq M$. 
\end{proof}

\begin{rem}\label{rem_connected}
We use the characterization of disconnected hypergaphs in terms of a quotient 
(see Lemma~\ref{lem_connected} below)  to prove the analogue of 
Theorem~\ref{tournament_ThB_Gallai} (see Theorem~\ref{Th-bis_Gallai} below). 
Recall that a hypergraph $H$ is {\em connected} if for distinct $v,w\in V(H)$, there exist a sequence $(e_0,\ldots,e_n)$ of edges of $H$, where $n\geq 0$, satisfying $v\in e_0$, $w\in e_n$, and (when $n\geq 1$) $e_i\cap e_{i+1}\neq\emptyset$ for every $0\leq i\leq n-1$. 
Given a hypergraph $H$, a maximal connected subhypergraph of $H$ is called a {\em component} of $H$. 

\begin{nota}\label{set_of_comp}
Given a hypergraph $H$, the set of the components of $H$ is denoted by $\mathfrak{C}(H)$. 
\end{nota}

Let $H$ be a hypergraph. 
For each component $C$ of $H$, $V(C)$ is a module of $H$. 
Thus, $\{V(C):C\in\mathfrak{C}(H)\}$ is a modular partition of $H$. 
Furthermore, for each component $C$ of $H$, $V(C)$ is a strong module of $H$. 
We conclude the remark with the following result. 

\begin{lem}\label{lem_connected}
Given a hypergraph $H$ with $v(H)\geq 2$, the following assertions are equivalent
\begin{enumerate}
\item $H$ is disconnected;
\item $H$ admits a modular bipartition $P$ such that $|P|\geq 2$ and $H/P$ is empty;
\item $\Pi(H)=\{V(C):C\in\mathfrak{C}(H)\}$, $|\Pi(H)|\geq 2$, and $H/\Pi(H)$ is empty.
\end{enumerate}
\end{lem}
\end{rem}

Let $H$ be a hypergraph such that $v(H)\geq 2$. 
Because of the maximality of the elements of $\Pi(H)$ (see Notation~\ref{nota_Pi_hyper}), it follows from the second assertion of Proposition~\ref{prop2_partition} that all the strong modules of $H/\Pi(H)$ are trivial. 
To prove Theorem~\ref{Thbis_Gallai}, we establish the following result, which is the 
analogue of Theorem~\ref{tournament_ThB_Gallai}. 

\begin{thm}\label{Th-bis_Gallai}
Given a hypergraph $H$, all the strong modules of $H$ are trivial if and only if $H$ is an empty hypergraph, a prime hypergraph or a complete graph. 
\end{thm}

\begin{proof}
Clearly, if $H$ is an empty hypergraph, a prime hypergraph or a complete graph, then 
all the strong modules of $H$ are trivial. 

To demonstrate the converse, we prove the following. 
Given a hypergraph $H$, if all the strong modules of $H$ are trivial, and $H$ is decomposable, then 
$H$ is an empty hypergraph or a complete graph. 

To begin, we show that $H$ admits a modular bipartition. 
Since $H$ is decomposable, we can consider a maximal nontrivial module $M$ of $H$ under inclusion. 
Since $M$ is a nontrivial module of $H$, $M$ is not strong. 
Consequently, there exists a module $N$ of $H$ such that $M\cap N\neq\emptyset$, $M\setminus N\neq\emptyset$ and $N\setminus M\neq\emptyset$. 
Since $M\cap N\neq\emptyset$, $M\cup N$ is a module of $H$ by Lemma~\ref{lem_union}. 
Clearly, $M\subsetneq M\cup N$ because $N\setminus M\neq\emptyset$. 
Since $M$ is a maximal nontrivial module of $H$, $M\cup N$ is a trivial module of $H$, so 
$M\cup N=V(H)$. 
Since $M\setminus N\neq\emptyset$, $N\setminus M$ is a module of $H$ by Lemma~\ref{lem_moins}. 
But, $N\setminus M=V(H)\setminus M$ because $M\cup N=V(H)$. 
It follows that $\{M,V(H)\setminus M\}$ is a modular bipartition of $H$. 

We have $H/\{M,V(H)\setminus M\}$ is an empty hypergraph or a complete graph. 
We distinguish the following two cases. 
\begin{enumerate}
\item Suppose that $H/\{M,V(H)\setminus M\}$ is an empty hypergraph. 
We prove that $H$ is an empty hypergraph. 
By Lemma~\ref{lem_connected}, $H$ is disconnected. 
Let $C\in\mathfrak{C}(H)$. 
As recalled in Remark~\ref{rem_connected}, $V(C)$ is a strong module of $H$. 
By hypothesis, $V(C)$ is trivial. 
Since $H$ is disconnected, $V(C)\subsetneq V(H)$. 
It follows that $v(C)=1$. 
Therefore, $H$ is isomorphic to $H/\{V(C):C\in\mathfrak{C}(H)\}$. 
It follows from Lemma~\ref{lem_connected} that $H$ is empty. 
\item Suppose that $H/\{M,V(H)\setminus M\}$ is a complete graph. 
We prove that $H$ is a complete graph. 
Consider the graph $H^c$ defined on $V(H)$ by 
\begin{equation}\label{E0_Th-bis_Gallai}
E(H^c)=(E(H)\setminus\binom{V(H)}{2})\cup(\binom{V(H)}{2}\setminus E(H)).
\end{equation}
It is easy to verify that $H$ and $H^c$ share the same modules. 
Therefore, they share the same strong modules. 
Consequently, all the strong modules of $H^c$ are trivial, $H^c$ is decomposable, and 
$\{M,V(H)\setminus M\}$ is a modular bipartition of $H$. 
Since $H/\{M,V(H)\setminus M\}$ is a complete graph, $H^c/\{M,V(H)\setminus M\}$ is empty. 
It follows from the first case that $H^c$ is empty. 
Hence $E(H^c)=\emptyset$, and it follows from \eqref{E0_Th-bis_Gallai} that $E(H)=\binom{V(H)}{2}$. \qedhere
\end{enumerate}
\end{proof}

\begin{proof}[Proof of Theorem~\ref{Thbis_Gallai}]
For a contradiction, suppose that $H/\Pi(H)$ admits a nontrivial strong module $\mathcal{S}$. 
By the second assertion of Proposition~\ref{prop2_partition}, $\cup\mathcal{S}$ is a strong module of $H$. 
Given $X\in\mathcal{S}$, we obtain $X\subsetneq\cup\mathcal{S}\subsetneq V(H)$, which contradicts the maximality of $X$. 
Consequently, all the strong modules of $H/\Pi(H)$ are trivial. 
To conclude, it suffices to apply Theorem~\ref{Th-bis_Gallai} to $H/\Pi(H)$. 
\end{proof}

\begin{defi}\label{tree_hypergraph}
Let $H$ be a hypergraph. 
As for tournaments (see Definition~\ref{tree_tournaments}), the set of the nonempty strong modules of $H$ is denoted by $\mathscr{D}(H)$. Clearly, $\mathscr{D}(H)$ endowed with inclusion is a tree called the {\em modular decomposition tree} of $H$. 
For convenience, set 
$$\mathscr{D}_{\geq 2}(H)=\{X\in\mathscr{D}(H):|X|\geq 2\}.$$

Moreover, we associate with each $X\in\mathscr{D}_{\geq 2}(H)$, the label 
$\varepsilon_H(X)$ defined as follows
\begin{equation*}
\varepsilon_H(X)=\ 
\begin{cases}
\text{$\triangle$ if $H[X]/\Pi(H[X])$ is prime},\\
\text{$\medcircle$ if $H[X]/\Pi(H[X])$ is empty}\\
\text{or}\\
\text{$\bullet$ if $H[X]/\Pi(H[X])$ is a complete graph.}
\end{cases}
\end{equation*}
\end{defi}

To conclude, we prove the analogue of Proposition~\ref{prop_strong_tournaments} for hypergraphs.

\begin{prop}\label{prop_strong_hypergraphs}
Given a hypergraph $H$, consider a strong module $M$ of $H$. 
For every $N\subseteq M$, the following two assertions are equivalent
\begin{enumerate}
\item $N$ is a strong module of $H$;
\item $N$ is a strong module of $H[M]$.
\end{enumerate}
\end{prop}

\begin{proof}
Let $N$ be a subset of $M$. 
To begin, suppose that $N$ is a strong module of $H$. 
Since $N$ is a module of $H$, $N$ is a module of $H[M]$ by Lemma~\ref{lem_(A2)}. 
To show that $N$ is a strong module of $H[M]$, consider a module $X$ of $H[M]$ such that 
$N\cap X\neq\emptyset$. 
Since $M$ is a module of $H$, $X$ is a module of $H$ by Lemma~\ref{lem_(A3)}. 
Since $N$ is a strong module of $H$, we obtain $N\subseteq X$ or $X\subseteq N$. 

Conversely, suppose that $N$ is a strong module of $H[M]$. 
Since $M$ is a module of $H$, $N$ is a module of $H$ by Lemma~\ref{lem_(A3)}. 
To show that $N$ is a strong module of $H$, consider a module $X$ of $H$ such that 
$N\cap X\neq\emptyset$. 
We have $M\cap X\neq\emptyset$ because $N\subseteq M$. 
Since $M$ is a strong module of $H$, we obtain $M\subseteq X$ or $X\subseteq M$. 
In the first instance, we get $N\subseteq M\subseteq X$. 
Hence, suppose that $X\subseteq M$. 
By Lemma~\ref{lem_(A2)}, $X$ is a module of $H[M]$. 
Since $N$ is a strong module of $H[M]$ and $N\cap X\neq\emptyset$, 
we obtain $N\subseteq X$ or $X\subseteq N$. 
\end{proof}

\section{Realization and decomposability}\label{real_dec}

Consider a realizable 3-uniform hypergraph. 
Let $T$ be a realization of $H$. A module of $T$ is clearly a module of $H$, but the converse is false. 
Nevertheless, we have the following result (see Proposition~\ref{P1_real_dec}). 
We need the following notation. 

\begin{nota}\label{tree_tilde}
Let $H$ be a 3-uniform hypergraph. 
For $W\subseteq V(H)$ such that $W\neq\emptyset$, $\widetilde{W}^H$ denotes the intersection of the strong modules of $H$ containing $W$. Note that $\widetilde{W}^H$ is the smallest strong module of $H$ containing $W$. 
\end{nota}

\begin{prop}\label{P1_real_dec}
Let $H$ be a realizable 3-uniform hypergraph. 
Consider a realization $T$ of $H$. 
Let $M$ be a module of $H$. 
Suppose that $M$ is not a module of $T$, and set 
$$\righthalfcap_T M=\{v\in V(H)\setminus M:\text{$M$ is not a module of $T[M\cup\{v\}]$}\}.$$
The following four assertions hold
\begin{enumerate}
\item $M\cup(\righthalfcap_T M)$ is a module of $T$;
\item $M$ is not a strong module of $H$;
\item $M\cup(\righthalfcap_T M)\subseteq\widetilde{M}^H$;
\item $\varepsilon_H(\widetilde{M}^H)=\medcircle$ and 
$|\Pi(H[\widetilde{M}^H])|\geq 3$. 
\end{enumerate}
\end{prop}

\begin{proof}
Since $M$ is not a module of $T$, we have $\righthalfcap_T M\neq\emptyset$. 
Let $v\in\righthalfcap_T M$. 
Since $M$ is not a module of $T[M\cup\{v\}]$, we obtain 
\begin{equation}\label{E0_P1_real_dec}
\begin{cases}
N_T^-(v)\cap M\neq\emptyset\\
\text{and}\\
N_T^+(v)\cap M\neq\emptyset.
\end{cases}
\end{equation}
Furthermore, consider $v^-\in N_T^-(v)\cap M$ and $v^+\in N_T^+(v)\cap M$. 
Since $M$ is a module of $H$, $v^-vv^+\not\in E(H)$. 
Hence $v^-vv^+\not\in E(C_3(T))$. 
Since $v^-v,vv^+\in A(T)$, we get $v^-v^+\in A(T)$. 
Therefore, for each $v\in\righthalfcap_T M$, we have 
\begin{equation}\label{E1_P1_real_dec}
\text{for $v^-\in N_T^-(v)\cap M$ and $v^+\in N_T^+(v)\cap M$, $v^-v^+\in A(T)$. }
\end{equation}

Now, consider $v,w\in\righthalfcap_T M$ such that $vw\in A(T)$. 
Let $v^-\in N_T^-(v)\cap M$. 
Suppose for a contradiction that $v^-\in N_T^+(w)\cap M$. 
We get $v^-vw\in E(C_3(T))$, and hence $v^-vw\in E(H)$. 
Since $M$ is a module of $H$, we obtain $\mu vw\in E(H)$ for every $\mu\in M$. 
Thus, since $vw\in A(T)$, $\mu v\in A(T)$ for every $\mu\in M$. 
Therefore, $M\subseteq N_T^-(v)$, so $N_T^+(v)\cap M=\emptyset$, which contradicts 
\eqref{E0_P1_real_dec}. 
It follows that for $v,w\in\righthalfcap_T M$, we have
\begin{equation}\label{E2_P1_real_dec}
\text{if $vw\in A(T)$, then $N_T^-(v)\cap M\subseteq N_T^-(w)\cap M$.} 
\end{equation}

For the first assertion, set 
\begin{equation*}
\begin{cases}
M^-=\{v\in V(H)\setminus M:vm\in A(T)\ \text{for every}\ m\in M\}\\
\text{and}\\
M^+=\{v\in V(H)\setminus M:mv\in A(T)\ \text{for every}\ m\in M\}.
\end{cases}
\end{equation*}
Note that $\{M^-,M,\righthalfcap_T M,M^+\}$ is a partition of $V(H)$. 
Let $m^-\in M^-$ and $v\in\righthalfcap_T M$. 
By \eqref{E0_P1_real_dec}, there exist 
$v^-\in N_T^-(v)\cap M$ and $v^+\in N_T^+(v)\cap M$. 
Suppose for a contradiction that $vm^-\in A(T)$. 
We get $vm^-v^-\in E(C_3(T))$. 
Hence $vm^-v^-\in E(H)$. 
Since $m^-v^+,vv^+\in A(T)$, we have $vm^-v^+\not\in E(C_3(T))$. 
Thus $vm^-v^+\not\in E(H)$, which contradicts the fact that $M$ is a module of $H$. 
It follows that $m^-v\in A(T)$ for any $m^-\in M^-$ and $v\in\righthalfcap_T M$. 
Similarly, $vm^+\in A(T)$ for any $m^+\in M^+$ and $v\in\righthalfcap_T M$. 
It follows that $M\cup(\righthalfcap_T M)$ is a module of $T$. 

For the second assertion, consider $v\in(\righthalfcap_T M)$. 
Set 
\begin{equation}\label{E3_P1_real_dec}
N_v=(N_T^-(v)\cap M)\cup\{w\in(\righthalfcap_T M):N_T^-(w)\cap M\subseteq N_T^-(v)\cap M\}.
\end{equation}
We show that $N_v$ is a module of $T$. 
If $m^-\in M^-$, then $m^-n\in A(T)$ for every $n\in N_v$ because 
$M\cup(\righthalfcap_T M)$ is a module of $T$. 
Similarly, if $m^+\in M^+$, then $nm^+\in A(T)$ for every $n\in N_v$. 
Now, consider $m\in M\setminus N_v$. 
We get $m\in M\setminus N_T^-(v)$. 
Therefore, we have $m\in M\setminus N_T^-(w')$ for every 
$w'\in\{w\in(\righthalfcap_T M):N_T^-(w)\cap M\subseteq N_T^-(v)\cap M\}$. 
Thus, $m\in N_T^+(w')\cap M$ for every 
$w'\in\{w\in(\righthalfcap_T M):N_T^-(w)\cap M\subseteq N_T^-(v)\cap M\}$. 
Since $m\in N_T^+(v)\cap M$, it follows from \eqref{E1_P1_real_dec} that 
$v^-m\in A(T)$ for every $v^-\in N_T^-(v)\cap M$. 
Furthermore, since $m\in N_T^+(w')\cap M$ for every 
$w'\in\{w\in(\righthalfcap_T M):N_T^-(w)\cap M\subseteq N_T^-(v)\cap M\}$, we have 
$w'm\in A(T)$ for every 
$w'\in\{w\in(\righthalfcap_T M):N_T^-(w)\cap M\subseteq N_T^-(v)\cap M\}$. 
Therefore, we obtain $nm\in A(T)$ for every $n\in N_v$. 
Lastly, consider $u\in (\righthalfcap_T M)\setminus N_v$. 
We get $u\in (\righthalfcap_T M)$ and $N_T^-(u)\cap M\not\subseteq N_T^-(v)\cap M$. 
It follows from \eqref{E2_P1_real_dec} that $vu\in A(T)$. 
By \eqref{E2_P1_real_dec} again, we have 
$N_T^-(v)\cap M\subsetneq N_T^-(u)\cap M$. 
Thus $v^-u\in A(T)$ for each $v^-\in N_T^-(v)\cap M$. 
Let $w'\in\{w\in(\righthalfcap_T M):N_T^-(w)\cap M\subseteq N_T^-(v)\cap M\}$. 
We get $N_T^-(w')\cap M\subsetneq N_T^-(u)\cap M$. 
It follows from \eqref{E2_P1_real_dec} that $w'u\in A(T)$. 
Consequently, $N_v$ is a module of $T$ for each $v\in(\righthalfcap_T M)$. 
Hence, 
\begin{equation}\label{E3a_P1_real_dec}
\text{$N_v$ is a module of $H$ for each $v\in(\righthalfcap_T M)$.} 
\end{equation}
(We use \eqref{E3a_P1_real_dec} to prove the third assertion below.) 
Let $v\in(\righthalfcap_T M)$. 
Clearly, $v\in N_v\setminus M$. 
Moreover, it follows from \eqref{E0_P1_real_dec} that there exist $v^-\in N_T^-(v)\cap M$ and 
$v^+\in N_T^+(v)\cap M$. 
We get $v^-\in M\cap N_v$ and $v^+\in M\setminus N_v$. 
Since $N_v$ is a module of $H$, $M$ is not a strong module of $H$.

For the third assertion, consider $v\in(\righthalfcap_T M)$. 
As previously proved, $N_v$ is a module of $H$. 
Furthermore, by considering $v^-\in N_T^-(v)\cap M$ and 
$v^+\in N_T^+(v)\cap M$, we obtain $M\cap N_v\neq\emptyset$ and $M\setminus N_v\neq\emptyset$. 
Hence $\widetilde{M}^H\cap N_v\neq\emptyset$ and 
$\widetilde{M}^H\setminus N_v\neq\emptyset$. 
Since $\widetilde{M}^H$ is a strong module of $H$, we get $N_v\subseteq\widetilde{M}^H$. 
Thus $v\in\widetilde{M}^H$ for every $v\in(\righthalfcap_T M)$. 
Therefore $M\cup(\righthalfcap_T M)\subseteq\widetilde{M}^H$. 

For the fourth assertion, we prove that for each $v\in(\righthalfcap_T M)$, 
$$P_v=\{N_T^-(v)\cap M,N_T^+(v)\cap M,\righthalfcap_T M\}$$ is a modular partition of 
$H[M\cup(\righthalfcap_T M)]$. 
Let $v\in(\righthalfcap_T M)$. 
By \eqref{E1_P1_real_dec}, $N_T^-(v)\cap M$ and $N_T^+(v)\cap M$ are 
modules of $T[M]$. 
Thus, $N_T^-(v)\cap M$ and $N_T^+(v)\cap M$ are modules of $H[M]$. 
Since $M$ is a module of $H$, it follows from Lemma~\ref{lem_(A3)} that $N_T^-(v)\cap M$ and $N_T^+(v)\cap M$ are modules of $H$. 
By Lemma~\ref{lem_(A2)}, 
$N_T^-(v)\cap M$ and $N_T^+(v)\cap M$ are modules of $H[M\cup(\righthalfcap_T M)]$. 
Now, we prove that $\righthalfcap_T M$ is a module of $H[M\cup(\righthalfcap_T M)]$. 
It suffices to prove that there exists no $e\in E(H[M\cup(\righthalfcap_T M)])$ such that 
$e\cap(\righthalfcap_T M)\neq\emptyset$ and $e\cap M\neq\emptyset$. 
Indeed, suppose to the contrary that there exists $e\in E(H[M\cup(\righthalfcap_T M)])$ such that $e\cap(\righthalfcap_T M)\neq\emptyset$ and $e\cap M\neq\emptyset$. 
Since $M$ is a module of $H$, we get $|e\cap M|=1$ and 
$|e\cap(\righthalfcap_T M)|=2$. 
Therefore, there exist $v,w\in e\cap(\righthalfcap_T M)$ and $m\in e\cap M$ such that 
$vwm\in E(H)$. 
By replacing $v$ by $w$ if necessary, we can assume that $vw\in A(T)$. 
Since $H=C_3(T)$, we obtain $vw,wm,mv\in A(T)$, which contradicts \eqref{E2_P1_real_dec}. 
Therefore, $\righthalfcap_T M$ is a module of $H[M\cup(\righthalfcap_T M)]$. 
Consequently, $P_v=\{N_T^-(v)\cap M,N_T^+(v)\cap M,\righthalfcap_T M\}$ is a modular partition of $H[M\cup(\righthalfcap_T M)]$. 
Furthermore, given $v\in(\righthalfcap_T M)$, consider $v^-\in N_T^-(v)\cap M$ and $v^+\in N_T^+(v)\cap M$. 
It follows from \eqref{E1_P1_real_dec} that $v^-v^+v\not\in E(C_3(T))$, and hence 
$v^-v^+v\not\in E(H)$. 
Consequently, 
\begin{equation}\label{E3bis_P1_real_dec}
\text{$H[M\cup(\righthalfcap_T M)]/P_v$ is empty.}
\end{equation}

Since $M\cup(\righthalfcap_T M)$ is a module of $T$ by the first assertion above, 
$M\cup(\righthalfcap_T M)$ is a module of $H$. 
By Lemma~\ref{lem_(A2)}, 
$M\cup(\righthalfcap_T M)$ is a module of $H[\widetilde{M}^H]$. 
Given $v\in(\righthalfcap_T M)$, it follows from Lemma~\ref{lem_(A3)} that each element of $P_v$ is a module of $H[\widetilde{M}^H]$. 

Let $v\in(\righthalfcap_T M)$. 
For a contradiction, suppose that there exist $Y\in P_v$ and $X\in\Pi(H[\widetilde{M}^H])$ such that $Y\subsetneq X$. 
We get $X\cap(M\cup(\righthalfcap_T M))\neq\emptyset$. 
Since $M\cup(\righthalfcap_T M)$ is a module of $H[\widetilde{M}^H]$ and 
$X$ is a strong module of $H[\widetilde{M}^H]$, we have $M\cup(\righthalfcap_T M)\subseteq X$ or $X\subsetneq M\cup(\righthalfcap_T M)$. 
Furthermore, since $X$ is a strong module of $H[\widetilde{M}^H]$ and 
$\widetilde{M}^H$ is a strong module of $H$, it follows from 
Proposition~\ref{prop_strong_hypergraphs} that $X$ is a strong module of $H$. Since $X\subsetneq\widetilde{M}^H$, it follows from the minimality of $\widetilde{M}^H$ that we do not have $M\cup(\righthalfcap_T M)\subseteq X$. 
Therefore, $X\subsetneq M\cup(\righthalfcap_T M)$. 
Let $x\in X\setminus Y$. 
We have $x\in(M\cup(\righthalfcap_T M))\setminus Y$. 
Denote by $Y'$ the unique element of $P_v\setminus\{Y\}$ such that $x\in Y'$. 
Also, denote by $Z$ the unique element of $P_v\setminus\{Y,Y'\}$. 
We get $X\cap Y'\neq\emptyset$ and $Y\subseteq X\setminus Y'$. 
Since $X$ is a strong module of $H[\widetilde{M}^H]$, we get $Y'\subseteq X$. 
Since $X\subsetneq M\cup(\righthalfcap_T M)$, we obtain $X\cap Z=\emptyset$. 
Thus $X=Y\cup Y'$. 
Since $H[M\cup(\righthalfcap_T M)]/P_v$ is empty by \eqref{E3bis_P1_real_dec}, $\{Y,Z\}$ is a module of 
$H[M\cup(\righthalfcap_T M)]/P_v$. 
By the second assertion of Proposition~\ref{prop1_partition}, $Y\cup Z$ is a module of 
$H[M\cup(\righthalfcap_T M)]$. 
As previously seen, $M\cup(\righthalfcap_T M)$ is a module of $H[\widetilde{M}^H]$. 
By Lemma~\ref{lem_(A3)}, 
$Y\cup Z$ is a module of $H[\widetilde{M}^H]$, which contradicts the fact that 
$X$ is a strong module of $H[\widetilde{M}^H]$. 
Consequently, 
\begin{equation}\label{E4_P1_real_dec}
\text{for any $Y\in P_v$ and $X\in\Pi(H[\widetilde{M}^H])$, we do not have 
$Y\subsetneq X$}.
\end{equation}
Let $Y\in P_v$. 
Set $$Q_Y=\{X\in \Pi(H[\widetilde{M}^H]):X\cap Y\neq\emptyset\}.$$
For every $X\in Q_Y$, we have 
$Y\subsetneq X$ or $X\subseteq Y$ because $X$ is a strong module of $H[\widetilde{M}^H]$. 
By \eqref{E4_P1_real_dec}, we have $X\subseteq Y$. 
It follows that 
\begin{equation}\label{E5_P1_real_dec}
\text{for each $Y\in P_v$, we have $Y=\cup Q_Y$.}
\end{equation}
Therefore, $|\Pi(H[\widetilde{M}^H])|\geq|P_v|$, that is, $$|\Pi(H[\widetilde{M}^H])|\geq 3.$$
Finally, we prove that $H[\widetilde{M}^H]/\Pi(H[\widetilde{M}^H])$ is empty. 
Suppose that $M\cup(\righthalfcap_T M)\subsetneq\widetilde{M}^H$, and set 
$$Q_{M\cup(\righthalfcap_T M)}=\{X\in \Pi(H[\widetilde{M}^H]):X\cap(M\cup(\righthalfcap_T M))\neq\emptyset\}.$$ 
Since $M\cup(\righthalfcap_T M)$ is a module of $H[\widetilde{M}^H]$, it follows from 
the first assertion of Proposition~\ref{prop1_partition} that 
$Q_{M\cup(\righthalfcap_T M)}$ is a module of 
$H[\widetilde{M}^H]/\Pi(H[\widetilde{M}^H])$. 
Moreover, it follows from \eqref{E5_P1_real_dec} that 
$|Q_{M\cup(\righthalfcap_T M)}|\geq 3$. 
Since each element of $\Pi(H[\widetilde{M}^H])$ is a strong element of $\widetilde{M}^H$, we get $M\cup(\righthalfcap_T M)=\cup Q_{M\cup(\righthalfcap_T M)}$. 
Since $M\cup(\righthalfcap_T M)\subsetneq\widetilde{M}^H$, we obtain that 
$Q_{M\cup(\righthalfcap_T M)}$ is a nontrivial module of 
$H[\widetilde{M}^H]/\Pi(H[\widetilde{M}^H])$. 
Hence $H[\widetilde{M}^H]/\Pi(H[\widetilde{M}^H])$ is decomposable. 
It follows from Theorem~\ref{Thbis_Gallai} that 
$H[\widetilde{M}^H]/\Pi(H[\widetilde{M}^H])$ is empty. 
Lastly, suppose that $M\cup(\righthalfcap_T M)=\widetilde{M}^H$. 
Suppose also that there exists $Y\in P_v$ such that $|Q_Y|\geq 2$. 
As previously, we obtain that $Q_Y$ is a nontrivial module of 
$H[\widetilde{M}^H]/\Pi(H[\widetilde{M}^H])$, and hence 
$H[\widetilde{M}^H]/\Pi(H[\widetilde{M}^H])$ is empty. 
Therefore, suppose that $|Q_Y|=1$ for every $Y\in P_v$. 
By \eqref{E5_P1_real_dec}, $\Pi(H[\widetilde{M}^H])=P_v$. 
Hence $H[\widetilde{M}^H]/\Pi(H[\widetilde{M}^H])$ is empty by 
\eqref{E3bis_P1_real_dec}. 
\end{proof}

The next result is an easy consequence of Proposition~\ref{P1_real_dec}. 

\begin{cor}\label{C1_real_dec}
Consider a realizable 3-uniform hypergraph $H$, and a realization $T$ of $H$. 
The following two assertions are equivalent
\begin{itemize}
\item $H$ and $T$ share the same modules;
\item for each strong module $X$ of $H$ such that $|X|\geq 2$, we have 
$$\text{if $\varepsilon_H(X)=\medcircle$, then 
$|\Pi(H[X])|=2$.}$$
\end{itemize}
\end{cor}

\begin{proof}
To begin, suppose that $H$ and $T$ do not share the same modules. 
There exists a module $M$ of $H$, which is not a module of $T$. 
By the last assertion of Proposition~\ref{P1_real_dec}, we obtain 
$\varepsilon_H(\widetilde{M}^H)=\medcircle$ and 
$|\Pi(H[\widetilde{M}^H])|\geq 3$. 

Conversely, suppose that there exists a strong module $X$ of $H$, with $|X|\geq 2$, such that 
$\varepsilon_H(X)=\medcircle$ and 
$|\Pi(H[X])|\geq 3$. 
It follows from the second assertion of Proposition~\ref{P1_real_dec} that 
$X$ is a module of $T$. 
Observe that $T[X]$ realizes $H[X]$. 
Let $Y\in\Pi(H[X])$. 
Since $Y$ is a strong module of $H[X]$, 
it follows from the second assertion of Proposition~\ref{P1_real_dec} applied to $H[X]$ and $T[X]$ that $Y$ is a module of $T[X]$. 
Thus, $\Pi(H[X])$ is a modular partition of $T[X]$. 
Since $H[X]/\Pi(H[X])$ is empty, $T[X]/\Pi(H[X])$ is a linear order. 
Denote by $Y_{{\rm min}}$ the smallest element of $T[X]/\Pi(H[X])$. 
Similarly, denote by $Y_{{\rm max}}$ the largest element of $T[X]/\Pi(H[X])$. 
Since $H[X]/\Pi(H[X])$ is empty, $\{Y_{{\rm min}},Y_{{\rm max}}\}$ is a module of 
$H[X]/\Pi(H[X])$. 
By the second assertion of Proposition~\ref{prop1_partition}, 
$Y_{{\rm min}}\cup Y_{{\rm max}}$ is a module of $H[X]$. 
Since $X$ is a module of $H$, it follows from Lemma~\ref{lem_(A3)} that $Y_{{\rm min}}\cup Y_{{\rm max}}$ is a module of $H$. 
Lastly, since $|\Pi(H[X])|\geq 3$, there exists 
$Y\in\Pi(H[X])\setminus\{Y_{{\rm min}},Y_{{\rm max}}\}$. 
Since $Y_{{\rm min}}$ is the smallest element of $T[X]/\Pi(H[X])$ and 
$Y_{{\rm max}}$ is the largest one, we obtain 
$Y_{{\rm min}}Y,YY_{{\rm max}}\in A(T[X]/\Pi(H[X]))$. 
Therefore, for $y_{{\rm min}}\in Y_{{\rm min}}$, $y\in Y$ and 
$y_{{\rm max}}\in Y_{{\rm max}}$, we have 
$y_{{\rm min}}y,yy_{{\rm max}}\in A(T[X])$, and hence 
$y_{{\rm min}}y,yy_{{\rm max}}\in A(T)$. 
Consequently, $Y_{{\rm min}}\cup Y_{{\rm max}}$ is not a module of $T$. 
\end{proof}

Now, we prove Theorem~\ref{same_strong_modules} by using Proposition~\ref{P1_real_dec} 
and the following lemma. 

\begin{lem}\label{lem_same_strong_modules}
Consider a realizable 3-uniform hypergraph $H$. 
Given a realization $T$ of $H$, all the strong modules of $H$ are strong modules of $T$. 
\end{lem}

\begin{proof}
 Consider a strong module $M$ of $H$. 
By the second assertion of Proposition~\ref{P1_real_dec}, $M$ is a module of $T$. 
Let $N$ be a module of $T$ such that $M\cap N\neq\emptyset$. 
Since $N$ is a module of $T$, $N$ is a module of $H$. 
Furthermore, since $M$ is a strong module of $H$, we obtain 
$M\subseteq N$ or $N\subseteq M$. 
Therefore, $M$ is a strong module of $T$. 
\end{proof}

\begin{proof}[Proof of Theorem~\ref{same_strong_modules}]
By Lemma~\ref{lem_same_strong_modules}, 
all the strong modules of $H$ are strong modules of $T$. 

Conversely, consider a strong module $M$ of $T$. 
Since $M$ is a module of $T$, $M$ is a module of $H$. 
Let $N$ be a module of $H$ such that $M\cap N\neq\emptyset$. 
If $N$ is a module of $T$, then $M\subseteq N$ or $N\subseteq M$ because $M$ is a strong module of $T$. 
Hence suppose that $N$ is not a module of $T$. 
By the last assertion of Proposition~\ref{P1_real_dec}, 
\begin{equation}\label{E1_same_strong_modules}
\begin{cases}
\text{$H[\widetilde{N}^H]/\Pi(H[\widetilde{N}^H])$ is empty}\\
\text{and}\\
|\Pi(H[\widetilde{N}^H])|\geq 3. 
\end{cases}
\end{equation}
Since $M\cap N\neq\emptyset$, $M\cap\widetilde{N}^H\neq\emptyset$. 
Since $\widetilde{N}^H$ is a strong module of $H$, we get 
$\widetilde{N}^H\subseteq M$ or $M\subsetneq\widetilde{N}^H$. 
Clearly, if $\widetilde{N}^H\subseteq M$, then $N\subseteq M$. 
Thus, suppose that $$M\subsetneq\widetilde{N}^H.$$ 
We prove that $M\subseteq N$. 
By Lemma~\ref{lem_same_strong_modules}, 
$\widetilde{N}^H$ is a strong module of $T$. 
Since $M$ is a strong module of $T$, it follows from Proposition~\ref{prop_strong_tournaments} that $M$ is a strong module of $T[\widetilde{N}^H]$. 
For each $X\in\Pi(H[\widetilde{N}^H])$, $X$ is a strong module of $T[\widetilde{N}^H]$ by 
Lemma~\ref{lem_same_strong_modules}. 
Therefore, $\Pi(H[\widetilde{N}^H])$ is a modular partition of $T[\widetilde{N}^H]$. 
Clearly, $T[\widetilde{N}^H]/$ $\Pi(H[\widetilde{N}^H])$ is a realization of 
$H[\widetilde{N}^H]/\Pi(H[\widetilde{N}^H])$. 
Set $$Q_M=\{X\in\Pi(H[\widetilde{N}^H]):M\cap X\neq\emptyset\}.$$
By the first assertion of Proposition~\ref{prop_strong_quotient_tournaments}, $Q_M$ is a strong module of $T[\widetilde{N}^H]/$ $\Pi(H[\widetilde{N}^H])$. 
Since $H[\widetilde{N}^H]/\Pi(H[\widetilde{N}^H])$ is empty by \eqref{E1_same_strong_modules}, 
$T[\widetilde{N}^H]/\Pi(H[\widetilde{N}^H])$ is a linear order. 
By Theorem~\ref{tournament_ThB_Gallai}, $Q_M$ is a trivial module of 
$T[\widetilde{N}^H]/\Pi(H[\widetilde{N}^H])$. 
For a contradiction, suppose that $Q_M=\Pi(H[\widetilde{N}^H])$. 
Since $M$ is a strong module of $T[\widetilde{N}^H]$, we get $M=\widetilde{N}^H$, which contradicts $M\subsetneq\widetilde{N}^H$. 
It follows that $|Q_M|=1$. 
Hence there exists $X_M\in\Pi(H[\widetilde{N}^H])$ such that $$M\subseteq X_M.$$
Since $N$ is not a module of $T$, it follows from the second assertion of 
Proposition~\ref{P1_real_dec} that $N$ is not a strong module of $H$. 
Thus $N\subsetneq\widetilde{N}^H$. 
Set $$Q_N=\{X\in\Pi(H[\widetilde{N}^H]):N\cap X\neq\emptyset\}.$$
Since $\widetilde{N}^H$ is a strong module of $H$, it follows from 
Proposition~\ref{prop_strong_hypergraphs} that each element of 
$\Pi(H[\widetilde{N}^H])$ is a strong module of $H$. 
It follows from the minimality of $\widetilde{N}^H$ that $|Q_N|\geq 2$. 
Since each element of 
$\Pi(H[\widetilde{N}^H])$ is a strong module of $H$, we obtain $$N=\cup Q_N.$$
Since $M\cap N\neq\emptyset$, we get $X_M\in Q_N$. 
We obtain $M\subseteq X_M\subseteq N$. 
\end{proof}

Lastly, we establish Theorem~\ref{same_primality} by using Theorems~\ref{Thbis_Gallai} and \ref{same_strong_modules}.

\begin{proof}[Proof of Theorem~\ref{same_primality}]
Suppose that $H$ is prime. 
Since all the modules of $T$ are modules of $H$, $T$ is prime. 

Conversely, suppose that $T$ is prime. Hence, all the strong modules of $T$ are trivial. 
By Theorem~\ref{same_strong_modules}, all the strong modules of $H$ are trivial. 
We obtain $$\Pi(H)=\{\{v\}:v\in V(H)\}.$$
Thus, $H$ is isomorphic to $H/\Pi(H)$. 
It follows from Theorem~\ref{Thbis_Gallai} that $H$ 
is an empty hypergraph, a prime hypergraph or a complete graph. 
Since $T$ is prime, we have $E(C_3(T))\neq\emptyset$. 
Since $E(C_3(T))=E(H)$, there exists $e\in E(H)$ such that $|e|=3$. 
Therefore, $H$ is not an empty hypergraph, and $H$ is not a graph. 
It follows that $H$ is prime. 
\end{proof}

\section{Realizability of 3-uniform hypergraphs}\label{real}

The next proposition is useful to construct realizations from the modular decomposition tree of a realizable 3-uniform hypergraph. 
We need the following notation and remark. 

\begin{nota}
Let $H$ be a 3-uniform hypergraph. 
We denote by $\mathscr{R}(H)$ the set of the realizations of $H$. 
\end{nota}

\begin{rem}\label{rem_theta}
Let $H$ be a realizable 3-uniform hypergraph. 
Consider $T\in\mathscr{R}(H)$. 
It follows from Theorem~\ref{same_strong_modules} that $$\mathscr{D}(H)=\mathscr{D}(T).$$
By the same, for each $X\in\mathscr{D}_{\geq 2}(H)$, we have $$\Pi(H[X])=\Pi(T[X]).$$
Therefore, for each $X\in\mathscr{D}_{\geq 2}(H)$, $T[X]/\Pi(T[X])$ realizes $H[X]/\Pi(H[X])$, that is, 
$$T[X]/\Pi(T[X])\in\mathscr{R}(H[X]/\Pi(H[X])).$$ 
Set $$\mathscr{R}_\mathscr{D}(H)=\bigcup_{X\in\mathscr{D}_{\geq 2}(H)}\mathscr{R}(H[X]/\Pi(H[X])).$$
We denote by $\delta_H(T)$ the function 
\begin{equation*}
\begin{array}{rcl}
\mathscr{D}_{\geq 2}(H)&\longrightarrow&\mathscr{R}_\mathscr{D}(H)\\
Y&\longmapsto&T[Y]/\Pi(T[Y]).
\end{array}
\end{equation*}
Lastly, we denote by 
$\mathscr{Q}(H)$ the set of the functions $f$ from $\mathscr{D}_{\geq 2}(H)$ to 
$\mathscr{R}_\mathscr{D}(H)$ satisfying $f(Y)\in\mathscr{R}(H[Y]/\Pi(H[Y]))$ for each 
$Y\in\mathscr{D}_{\geq 2}(H)$. 
Under this notation, we obtain the function 
\begin{equation*}
\begin{array}{rccl}
 \delta_H:&\mathscr{R}(H)&\longrightarrow&\mathscr{Q}(H)\\
&T&\longmapsto&\delta_H(T).
\end{array}
\end{equation*}
\end{rem}

\begin{prop}\label{construction_realization}
For a 3-uniform hypergraph, $\delta_H$ is a bijection.  
\end{prop}

\begin{proof}
To begin, we show that $\delta_H$ is injective. 
Let $T$ and $T'$ be distinct realizations of $H$. 
There exist distinct $v,w\in V(H)$ such that $vw\in A(T)$ and $wv\in A(T')$. 
Consider $Z_v,Z_w\in\Pi(H[\widetilde{\{v,w\}}^H])$ (see Notation~\ref{tree_tilde}) such that $v\in Z_v$ and $w\in Z_w$. 
Since $\widetilde{\{v,w\}}^H$ is the smallest strong module of $H$ containing $\{v,w\}$, we obtain 
$Z_v\neq Z_w$. 
It follows from Theorem~\ref{same_strong_modules} that 
$\Pi(H[\widetilde{\{v,w\}}^H])=\Pi(T[\widetilde{\{v,w\}}^H])$ and $\Pi(H[\widetilde{\{v,w\}}^H])=\Pi(T'[\widetilde{\{v,w\}}^H])$. 
Since $vw\in A(T)$ and $wv\in A(T')$, we obtain 
\begin{equation*}
\begin{cases}
Z_vZ_w\in A(T[\widetilde{\{v,w\}}^H]/\Pi(T[\widetilde{\{v,w\}}^H]))\\
\text{and}\\
Z_wZ_v\in A(T'[\widetilde{\{v,w\}}^H]/\Pi(T'[\widetilde{\{v,w\}}^H])). 
\end{cases}
\end{equation*}
Consequently, $\delta_H(T)(\widetilde{\{v,w\}}^H)\neq\delta_H(T')(\widetilde{\{v,w\}}^H)$. 
Thus, $\delta_H(T)\neq\delta_H(T')$. 

Now, we prove that $\delta_H$ is surjective. 
Consider $f\in\mathscr{Q}(H)$, that is, $f$ is a function from $\mathscr{D}_{\geq 2}(H)$ to 
$\mathscr{R}_\mathscr{D}(H)$ satisfying $f(Y)\in\mathscr{R}(H[Y]/\Pi(H[Y]))$ for each 
$Y\in\mathscr{D}_{\geq 2}(H)$. 
We construct $T\in\mathscr{R}(H)$ such that $\delta_H(T)=f$ in the following manner. 
Consider distinct vertices $v$ and $w$ of $H$. 
Clearly, $\widetilde{\{v,w\}}^H$ is a strong module of $H$ such that $|\widetilde{\{v,w\}}^H|\geq 2$. 
There exist $Z_v,Z_w\in\Pi(H[\widetilde{\{v,w\}}^H])$ such that $v\in Z_v$ and $w\in Z_w$. 
Since $\widetilde{\{v,w\}}^H$ is the smallest strong module of $H$ containing $v$ and $w$, we obtain 
$Z_v\neq Z_w$. 
Set 
\begin{equation}\label{E1_construction_realization}
\begin{cases}
\text{$vw\in A(T)$ if $Z_vZ_w\in A(f(\widetilde{\{v,w\}}^H))$,}\\ 
\text{and}\\ 
\text{$wv\in A(T)$ if $Z_wZ_v\in A(f(\widetilde{\{v,w\}}^H))$. }
\end{cases}
\end{equation}
We obtain a tournament $T$ defined on $V(H)$. 

Lastly, we verify that $T$ realizes $H$. 
First, consider distinct vertices $u,v,w$ of $H$ such that $uvw\in E(H)$. 
There exist $Z_u,Z_v,Z_w\in\Pi(H[\widetilde{\{u,v,w\}}^H])$ such that $u\in Z_u$, $v\in Z_v$ and $w\in Z_w$. 
For a contradiction, suppose that $Z_u=Z_v$. 
Since $Z_u$ is a module of $H$ and $uvw\in E(H)$, we get $w\in Z_u$. 
Thus, $Z_u=Z_v=Z_w$, which contradicts the fact that $\widetilde{\{u,v,w\}}^H$ is the smallest strong module of $H$ containing $u,v$ and $w$. 
It follows that $Z_u\neq Z_v$. 
Similarly, we have $Z_u\neq Z_w$ and $Z_v\neq Z_w$. 
It follows that $Z_uZ_vZ_w\in E(H[\widetilde{\{u,v,w\}}^H]/\Pi(H[\widetilde{\{u,v,w\}}^H]))$. 
Since $f(\widetilde{\{u,v,w\}}^H)$ realizes $H[\widetilde{\{u,v,w\}}^H]/\Pi(H[\widetilde{\{u,v,w\}}^H])$, we obtain 
$Z_uZ_v,Z_vZ_w,$ $Z_wZ_u\in A(f(\widetilde{\{u,v,w\}}^H))$ or 
$Z_uZ_w,Z_wZ_v,Z_vZ_u\in A(f(\widetilde{\{u,v,w\}}^H))$. 
By exchanging $u$ and $v$ if necessary, assume that 
$$Z_uZ_v,Z_vZ_w,Z_wZ_u\in A(f(\widetilde{\{u,v,w\}}^H)).$$ 
Since $Z_u\neq Z_v$, we obtain $\widetilde{\{u,v\}}^H=\widetilde{\{u,v,w\}}^H$. 
Similarly, we have $\widetilde{\{u,w\}}^H=\widetilde{\{u,v,w\}}^H$ and $\widetilde{\{v,w\}}^H=\widetilde{\{u,v,w\}}^H$. 
It follows from \eqref{E1_construction_realization} that $uv,vw,wu\in A(T)$. 
Hence, $T[\{u,v,w\}]$ is a 3-cycle. 

Conversely, consider distinct vertices $u,v,w$ of $T$ such that $T[\{u,v,w\}]$ is a 3-cycle. 
There exist $Z_u,Z_v,Z_w\in\Pi(H[\widetilde{\{u,v,w\}}^H])$ such that $u\in Z_u$, $v\in Z_v$ and $w\in Z_w$. 
For a contradiction, suppose that $Z_u=Z_v$. 
Since $\widetilde{\{u,v,w\}}^H$ is the smallest strong module of $H$ containing $u,v$ and $w$, we obtain $Z_u\neq Z_w$. 
Therefore, we have $\widetilde{\{u,w\}}^H=\widetilde{\{u,v,w\}}^H$ and $\widetilde{\{v,w\}}^H=\widetilde{\{u,v,w\}}^H$. 
For instance, assume that $Z_uZ_w\in A(f(\widetilde{\{u,v,w\}}^H))$. 
It follows from \eqref{E1_construction_realization} that $uw,vw\in A(T)$, which contradicts the fact that $T[\{u,v,w\}]$ is a 3-cycle. 
Consequently, $Z_u\neq Z_v$. 
It follows that $\widetilde{\{u,v\}}^H=\widetilde{\{u,v,w\}}^H$. 
Similarly, we have $\widetilde{\{u,w\}}^H=\widetilde{\{u,v,w\}}^H$ and $\widetilde{\{v,w\}}^H=\widetilde{\{u,v,w\}}^H$. 
For instance, assume that $uv,vw,wu\in A(T)$. 
It follows from \eqref{E1_construction_realization} that 
$Z_uZ_v,Z_vZ_w,Z_wZ_u\in A(f(\widetilde{\{u,v,w\}}^H))$. 
Since $f(\widetilde{\{u,v,w\}}^H)$ realizes $H[\widetilde{\{u,v,w\}}^H]/\Pi(H[\widetilde{\{u,v,w\}}^H])$, we obtain 
$Z_uZ_vZ_w\in E(H[\widetilde{\{u,v,w\}}^H]/\Pi(H[\widetilde{\{u,v,w\}}^H]))$. 
It follows that $uvw\in E(H[\widetilde{\{u,v,w\}}^H])$, and hence $uvw\in E(H)$. 

Consequently, $T\in\mathscr{R}(H)$. 
Let $X\in\mathscr{D}_{\geq 2}(H)$. 
As seen at the beginning of Remark~\ref{rem_theta}, we have $\Pi(H[X])=\Pi(T[X])$, and $$T[X]/\Pi(T[X])\in\mathscr{R}(H[X]/\Pi(H[X])).$$ 
Consider distinct elements $Y$ and $Z$ of $\Pi(H[X])$. 
For instance, assume that $YZ\in A(T[X]/\Pi(T[X]))$. 
Let $v\in Y$ and $w\in Z$. 
We obtain $vw\in A(T)$. 
Moreover, we have $\widetilde{\{v,w\}}^H=X$ because $Y,Z\in\Pi(H[X])$ and $Y\neq Z$. 
It follows from \eqref{E1_construction_realization} that 
$YZ\in A(f(X))$. 
Therefore, 
\begin{equation}\label{E2_construction_realization}
T[X]/\Pi(T[X])=f(X). 
\end{equation}
Since \eqref{E2_construction_realization} holds for every $X\in\mathscr{D}_{\geq 2}(H)$, we have 
$\delta_H(T)=f$. 
\end{proof}

Theorem~\ref{realiza_dec} is an easy consequence of Proposition~\ref{construction_realization}.

\begin{proof}[Proof of Theorem~\ref{realiza_dec}]
Clearly, if $H$ is realizable, then $H[W]$ is also for every $W\subseteq V(H)$. Conversely, suppose that $H[W]$ is realizable for every $W\subseteq V(H)$ such that $H[W]$ is prime. 
We define an element $f$ of $\mathscr{Q}(H)$ as follows. 
Consider $Y\in\mathscr{D}_{\geq 2}(H)$. 
By Theorem~\ref{Thbis_Gallai}, $H[Y]/\Pi(H[Y])$ is empty or prime. 
First, suppose that $H[Y]/\Pi(H[Y])$ is empty. 
We choose for $f(Y)$ any linear order defined on $\Pi(H[Y])$. 
Clearly, $f(Y)\in\mathscr{R}(H[Y]/\Pi(H[Y]))$. 
Second, suppose that $H[Y]/\Pi(H[Y])$ is prime. 
Consider a transverse $W$ of $\Pi(H[Y])$ (see Definition~\ref{defi_transverse}). 
The function 
\begin{equation*}
\begin{array}{rccl}
\theta_W&W:&\longrightarrow&\Pi(H[Y])\\
&w&\longmapsto&\text{$Z$, where $w\in Z$,}
\end{array}
\end{equation*}
is an isomorphism from $H[W]$ onto $H[Y]/\Pi(H[Y])$. 
Thus, $H[W]$ is prime. 
By hypothesis, $H[W]$ admits a realization $T_W$. 
We choose for $f(Y)$ the unique tournament defined on $\Pi(H[Y])$ such that $\theta_W$ is an isomorphism from $T_W$ onto $f(Y)$. 
Clearly, $f(Y)\in\mathscr{R}(H[Y]/\Pi(H[Y]))$. 

By Proposition~\ref{construction_realization}, $(\delta_H)^{-1}(f)$ is a realization of $H$. 
\end{proof}

Theorem~\ref{realiza_dec} leads us to study the realization of prime and 3-uniform hypergraphs. 
We need to introduce the analogue of Defintion~\ref{nota_critical} for 3-uniform hypergraphs. 

\begin{defi}\label{nota_critical_hyper}
Given a prime and 3-uniform hypergraph $H$, a vertex $v$ of $H$ is {\em critical} if $H-v$ is decomposable. 
A prime and 3-uniform hypergraph is {\em critical} if all its vertices are critical. 
\end{defi}

For critical and 3-uniform hypergraphs, we obtain the following characterization, which is an immediate consequence of Theorems~\ref{same_primality} and~\ref{th1 sch trott}. 

\begin{thm}\label{realiza_critical}
Given a critical and 3-uniform hypergraph $H$, $H$ is realizable if and only if $v(H)$ is odd, and $H$ is isomorphic to $C_3(T_{v(H)})$, $C_3(U_{v(H)})$ or $C_3(W_{v(H)})$. 
\end{thm}

We pursue with the characterization of non critical, prime and 3-uniform hypergraphs that are realizable. 
We need the following notation. 

\begin{nota}\label{nota_realiza_non_critical}
Let $H$ be a 3-uniform hypergraph. Consider a vertex $x$ of $H$. 
Set $$V_x=V(H)\setminus\{x\}.$$
We denote by $G_x$ the graph defined on 
$V_x$ as follows. 
Given distinct elements $v$ and $w$ of $V_x$, 
$$\text{$vw\in E(G_x)$ if $xvw\in E(H)$\ \  (note that the graph $G_x$ is used in~\cite{FF84}) }.$$
Also, we denote by $I_x$ the set of the isolated vertices of $G_x$. 
Lastly, suppose that $H-x$ admits a realization $T_x$. 
Consider a bipartition $P$ of $V_x\setminus I_x$. 
Denote one element of $P$ by $X^-$, and the other one by $X^+$. 
Now, denote by $Y^-$ the set of $v\in I_x$ such that there exists a sequence $v_0,\ldots,v_n$ satisfying 
\begin{itemize}
\item $v_0\in X^-$;
\item $v_n=v$;
\item $v_1,\ldots,v_n\in I_x$;
\item for $i=0,\ldots,n-1$, $v_iv_{i+1}\in A(T_x)$. 
\end{itemize}
Dually, denote by $Y^+$ the set of $v\in I_x$ such that there exists a sequence $v_0,\ldots,v_n$ satisfying 
\begin{itemize}
\item $v_0\in X^+$;
\item $v_n=v$;
\item $v_1,\ldots,v_n\in I_x$;
\item for $i=0,\ldots,n-1$, $v_{i+1}v_i\in A(T_x)$. 
\end{itemize}
\end{nota}

\begin{thm}\label{realiza_non_critical}
Let $H$ be a non critical, prime, and 3-uniform hypergraph. Consider a vertex $x$ of $H$ such that $H-x$ is prime. Suppose that $H-x$ admits a realization $T_x$. 
Then, $H$ is realizable if and only if the following two assertions hold. 
\begin{enumerate}
\item[(M1)] There exists a bipartition $\{X^-,X^+\}$ of $V_x\setminus I_x$ satisfying
\begin{itemize}
\item for each component $C$ of $G_x$, with $v(C)\geq 2$, $C$ is bipartite with bipartition 
$\{X^-\cap V(C),X^+\cap V(C)\}$;
\item for $v^-\in X^-$ and $v^+\in X^+$, we have 
\begin{equation}\label{E0_realiza_non_critical}
\text{$v^-v^+\in E(G_x)$ if and only if $v^-v^+\in A(T_x)$.}
\end{equation}
\end{itemize}
\item[(M2)] We have $Y^-\cap Y^+=\emptyset$ and $Y^-\cup Y^+=I_x$. Furthermore, for $x^-\in X^-$, $x^+\in X^+$, $y^-\in Y^-$ and $y^+\in Y^+$, we have $y^+x^-,x^+y^-,y^+y^-\in A(T_x)$. 
\end{enumerate}
Moreover, if $H$ is realizable, then there exists a unique realization $T$ of $H$ such that $T-x=T_x$. 
Precisely, suppose that there exists a realization $T$ of $H$ such that $T-x=T_x$. 
For $x^-\in X^-$, $x^+\in X^+$, $y^-\in Y^-$ and $y^+\in Y^+$, we have $xx^-,x^+x,xy^-,y^+x\in A(T)$.
\end{thm}

\begin{proof}
To begin, suppose that $H$ admits a realization $T$. 
Clearly, $T-x$ is a realization of $H-x$. 
Since $T_x$ is a realization of $H-x$, we have $C_3(T_x)=C_3(T-x)$. 
Since $H-x$ is prime, it follows from Theorem~\ref{same_primality} that $T-x$ is prime as well. 
Since $C_3(T_x)=C_3(T-x)$, it follows from Theorem~\ref{thm_bilt} that $T_x=T-x$ or $(T-x)^\star$. 
By exchanging $T$ and $T^\star$ if necessary, we can assume that $$T_x=T-x.$$

We show that $G_x$ is bipartite. 
Consider a sequence $(v_0,\ldots,v_{2n})$ of distinct elements of $V_x$, where $n\geq 2$,  such that 
$v_iv_{i+1}\in E(G_x)$ for every $0\leq i\leq 2n-1$. 
For instance, assume that $v_0v_1\in A(T_x)$. Hence $v_0v_1\in A(T)$. 
Since $v_0v_1\in E(G_x)$, we have $xv_0v_1\in E(H)$. 
Since $T$ is a realization of $H$, with $v_0v_1\in A(T)$, we obtain 
$xv_0,v_1x\in A(T)$. 
Since $v_1v_2\in E(G_x)$, we have $xv_1v_2\in E(H)$. 
Since $T$ is a realization of $H$, with $v_1x\in A(T)$, we obtain 
$xv_2,v_2v_1\in A(T)$. 
By continuing this process, we obtain 
\begin{equation*}
\begin{cases}
xv_0,xv_2,\ldots,xv_{2n}\in A(T)\\
\text{and}\\
v_1x,\ldots,v_{2n-1}x\in A(T).
\end{cases}
\end{equation*}
Thus $xv_0,xv_{2n}\in A(T)$, and hence $T[\{x,v_0,v_{2n}\}]$ is not a 3-cycle. 
It follows that $xv_0v_{2n}\not\in E(H)$, so $v_0v_{2n}\not\in E(G_x)$. 
Therefore, $G_x$ does not contain odd cycles. 

For Assertion (M1), consider a component $C$ of $G_x$ such that $v(C)\geq 2$. 
Consider distinct vertices $c_0,c_1,c_2$ of $C$ such that $c_0c_1,c_1c_2\in E(G_x)$. 
We show that 
\begin{equation}\label{E1_realiza_non_critical}
\begin{cases}
c_0c_1,c_2c_1\in A(T_x)\\
\text{or}\\
c_1c_0,c_1c_2\in A(T_x).
\end{cases}
\end{equation}
Otherwise, suppose that $c_0c_1,c_1c_2\in A(T_x)$. 
Since $c_0c_1,\in E(G_x)$, $T[\{x,c_0,c_1\}]$ is a 3-cycle. 
Hence, $xc_0,c_1x\in A(T_x)$ because $c_0c_1,\in A(T_x)$. 
We obtain $c_1c_2,c_1x\in A(T_x)$. 
Thus, $T[\{x,c_1,c_2\}]$ is not a 3-cycle, which contradicts $c_1c_2\in E(G_x)$. 
It follows that \eqref{E1_realiza_non_critical} holds. 
Now, denote by $V(C)^-$ the set of the vertices $c^-$ of $C$ such that there exists $c^+\in V(C)$ satisfying 
$c^-c^+\in E(G_x)$ and $c^-c^+\in A(T_x)$. 
Dually, denote by $V(C)^+$ the set of the vertices $c^+$ of $C$ such that there exists $c^-\in V(C)$ satisfying 
$c^-c^+\in E(G_x)$ and $c^-c^+\in A(T_x)$. 
Since $C$ is a component of $G_x$, we have $V(C)=V(C)^-\cup V(C)^+$. 
Moreover, it follows from \eqref{E1_realiza_non_critical} that $$V(C)^-\cap V(C)^+=\emptyset.$$
Also, it follows from the definition of $V(C)^-$ and $V(C)^+$ that $V(C)^-$ and $V(C)^+$ are stable subsets of $C$. 
Therefore, $C$ is bipartite with bipartition $\{V(C)^-,V(C)^+\}$. 
Set 
$$X^-=\bigcup_{C\in\mathfrak{C}(G_x)}V(C)^-\ \text{and}\ X^+=\bigcup_{C\in\mathfrak{C}(G_x)}V(C)^+\ \text{(see Notation~\ref{set_of_comp})}.$$
Clearly, $\{X^-,X^+\}$ is a bipartition of $V_x\setminus I_x$. 
Consider again a component $C$ of $G_x$ such that $v(C)\geq 2$. 
Since $V(C)^-=X^-\cap V(C)$ and $V(C)^+=X^+\cap V(C)$, 
$C$ is bipartite with bipartition $\{X^-\cap V(C),X^+\cap V(C)\}$. 
To prove that \eqref{E0_realiza_non_critical} holds, consider $v^-\in X^-$ and $v^+\in X^+$. 
First, suppose that $v^-v^+\in E(G_x)$. 
Hence, $v^-$ and $v^+$ belong to the same component of $G_x$. 
Denote it by $C$. 
We obtain $v^-\in V(C)^-$ and $v^+\in V(C)^+$. 
By definition of $V(C)^-$, there exists $c^+\in V(C)$ such that $v^-c^+\in E(G_x)$ and $v^-c^+\in A(T_x)$. 
Since $v^-v^+\in E(G_x)$, it follows from \eqref{E1_realiza_non_critical} that $v^-v^+\in A(T_x)$. 
Second, suppose that $v^-v^+\in A(T_x)$. 
Since $v^-\in V(C)^-$, 
there exists $c^+\in V(C)$ such that $v^-c^+\in E(G_x)$ and $v^-c^+\in A(T_x)$. 
It follows that $xv^-, c^+x\in A(T)$. 
Similarly, since $v^+\in V(C)^+$, there exists $c^-\in V(C)$ such that $xc^-,v^+x\in A(T)$. 
We obtain $xv^-,v^-v^+,v^+x\in A(T)$. 
Thus, $T[\{x,v^-,v^+\}]$ is a 3-cycle, so $v^-v^+\in E(G_x)$.

For Assertion (M2), consider $x^-\in X^-$. 
Denote by $C$ the component of $G_x$ such that $x^-\in V(C)$. 
We have $x^-\in V(C)^-$. 
Therefore, there exists $c^+\in V(C)$ satisfying 
$x^-c^+\in E(G_x)$ and $x^-c^+\in A(T_x)$. 
Since $T[\{x,x^-,c^+\}]$ is a 3-cycle and $x^-c^+\in A(T)$, we get $xx^-\in A(T)$. 
Hence, 
\begin{equation}\label{E2_realiza_non_critical}
xx^-\in A(T)\ \text{for every $x^-\in X^-$. }
\end{equation}
Dually, we have 
\begin{equation}\label{E3_realiza_non_critical}
x^+x\in A(T)\ \text{for every $x^+\in X^+$. }
\end{equation}
Now, 
consider $y^-\in Y^-$. 
There exists a sequence $v_0,\ldots,v_n$ satisfying $v_0\in X^-$, $v_n=y^-$, $v_1,\ldots,v_n\in I_x$, and 
$v_iv_{i+1}\in A(T_x)$ for $i=0,\ldots,n-1$. 
We show that $xv_i\in A(T_x)$ by induction on $i=0,\ldots,n$. 
By \eqref{E2_realiza_non_critical}, this is the case when $i=0$. 
Consider $i\in\{0,\ldots,n-1\}$ and suppose that $xv_i\in A(T_x)$. 
Since $v_{i+1}\in I_x$, $v_iv_{i+1}\not\in E(G_x)$. 
Thus $T[\{x,v_i,v_{i+1}\}]$ is a linear order. 
Since $xv_i,v_iv_{i+1}\in A(T_x)$, we obtain $xv_{i+1}\in A(T_x)$. 
It follows that 
\begin{equation}\label{E4_realiza_non_critical}
xy^-\in A(T)
\end{equation} 
for every $y^-\in Y^-$. 
Dually, we have 
\begin{equation}\label{E5_realiza_non_critical}
y^+x\in A(T)
\end{equation} 
for every $y^+\in Y^+$. 
It follows from \eqref{E4_realiza_non_critical} and \eqref{E5_realiza_non_critical} that $Y^-\cap Y^+=\emptyset$. 
By definition of $Y^-$ and $Y^+$, $Y^-\subseteq I_x$ and $Y^+\subseteq I_x$. 
Set $$W=I_x\setminus(Y^-\cup Y^+).$$
Let $w\in W$. 
Since $w\not\in Y^-$, we have $wz^-\in A(T_x)$ for every $z^-\in X^-\cup Y^-$. 
Therefore, 
\begin{equation}\label{E6_realiza_non_critical}
wz^-\in A(T_x)
\end{equation} 
for $w\in W$ and $z^-\in X^-\cup Y^-$. 
Dually, 
\begin{equation}\label{E7_realiza_non_critical}
z^+w\in A(T_x)
\end{equation} 
for $w\in W$ and $z^+\in X^+\cup Y^+$. 
It follows from \eqref{E2_realiza_non_critical}, \eqref{E3_realiza_non_critical}, \eqref{E4_realiza_non_critical}, 
\eqref{E5_realiza_non_critical}, \eqref{E6_realiza_non_critical} and \eqref{E7_realiza_non_critical} that  
$\{x\}\cup W$ is a module of $T$. 
Since $H$ is prime, it follows from Theorem~\ref{same_primality} that $T$ is prime as well. 
Therefore, $W=\emptyset$, so $$Y^-\cup Y^+=I_x.$$
It follows also from \eqref{E2_realiza_non_critical}, \eqref{E3_realiza_non_critical}, \eqref{E4_realiza_non_critical} and \eqref{E5_realiza_non_critical} that $T$ is the unique realization of $H$ such that $T-x=T_x$. 
To conclude, consider $x^-\in X^-$, $x^+\in X^+$, $y^-\in Y^-$ and $y^+\in Y^+$. 
Since $y^-\not\in Y^+$, $x^+y^-\in A(T_x)$. 
Dually, we have $y^+x^-\in A(T_x)$. 
It follows from \eqref{E4_realiza_non_critical} and \eqref{E5_realiza_non_critical} that 
$y^+x,xy^-\in A(T_x)$. 
Since $y^-,y^+\in I_x$, $y^-y^+\not\in E(G_x)$. 
Thus, $T[\{x,y^-,y^+\}]$ is a linear order. 
Consequently, we have $y^+y^-\in A(T_x)$. 

Conversely, suppose that Assertions (M1) and (M2) hold. 
Let $T$ be the tournament defined on $V(H)$ by 
\begin{equation}\label{E8_realiza_non_critical}
\begin{cases}
T-x=T_x,\\
\text{for every $z^-\in X^-\cup Y^-$, $xz^-\in A(T)$,}\\
\text{and}\\
\text{for every $z^+\in X^+\cup Y^+$, $z^+x\in A(T)$.} 
\end{cases}
\end{equation}
We verify that $T$ is a realization of $H$. 
Since $T_x$ realizes $H-x$, it suffices to verify that for distinct $v,w\in V_x$, $vw\in E(G_x)$ if and only if 
$T[\{x,v,w\}]$ is a 3-cycle. 
Hence, consider distinct $v,w\in V_x$. 
First, suppose that $vw\in E(G_x)$. 
Denote by $C$ the component of $G_x$ containing $v$ and $w$. 
Since Assertion (M1) holds, $C$ is bipartite with bipartition 
$\{X^-\cap V(C),X^+\cap V(C)\}$. 
By exchanging $v$ and $w$ if necessary, we can assume that $v\in X^-\cap V(C)$ and $w\in X^+\cap V(C)$. 
It follows from \eqref{E0_realiza_non_critical} that $vw\in A(T_x)$. 
Furthermore, it follows from \eqref{E8_realiza_non_critical} that $xv\in A(T)$ and $wx\in A(T)$. 
Therefore, $T[\{x,v,w\}]$ is a 3-cycle. 
Second, suppose that $T[\{x,v,w\}]$ is a 3-cycle. 
By exchanging $v$ and $w$ if necessary, we can assume that $vw,wx,xv\in A(T)$. 
It follows from \eqref{E8_realiza_non_critical} that $v\in X^-\cup Y^-$ and $w\in X^+\cup Y^+$. 
Moreover, since Assertion (M2) holds and $vw\in A(T_x)$, we obtain $v\in X^-$ and $w\in X^+$. 
It follows from \eqref{E0_realiza_non_critical} that $vw\in E(G_x)$. 
\end{proof}

We conclude by counting the number of realizations of a realizable 3-uniform hypergraph. 
This counting is an immediate consequence of Proposition~\ref{construction_realization}. 
We need the following notation. 

\begin{nota}
Let $H$ be a 3-uniform hypergraph. 
Set 
\begin{equation*}
\begin{cases}
\mathscr{D}_\triangle(H)=\{X\in\mathscr{D}_{\geq 2}(H):\varepsilon_H(X)=\triangle\}\ \text{(see Definition~\ref{tree_hypergraph})}\\
\text{and}\\
\mathscr{D}_\medcircle(H)=\{X\in\mathscr{D}_{\geq 2}(H):\varepsilon_H(X)=\medcircle\}.
\end{cases}
\end{equation*}
\end{nota}

\begin{cor}\label{counting_realizations}
For a realizable 3-uniform hypergraph, we have 
\begin{equation*}
|\mathscr{R}(H)|=2^{|\mathscr{D}_\triangle(H)|}\times\prod_{X\in\mathscr{D}_\medcircle(H)}|\Pi(H[X])|!\hspace{2mm}.
\end{equation*}
\end{cor}

\end{document}